\documentclass[11pt,a4paper,leqno]{amsart}
\usepackage{version}
\usepackage[utf8]{inputenc}
\usepackage{amsmath,amssymb,amsfonts,amsthm,yfonts}
\usepackage{calrsfs}
\usepackage{esint}
\usepackage{hyperref}
\usepackage{mathtools}
\usepackage{bbm}
\usepackage{enumitem}
\usepackage{graphicx}
\usepackage{caption}
\usepackage{subcaption}
\usepackage{epsfig}
\usepackage{float}
\usepackage{latexsym}

\newcommand*\mathinhead[2]{\texorpdfstring{$#1$}{#2}}
\newcommand*{\dd}{\mathop{}\!\mathrm{d}}

\newcommand*{\ox}{\overline{x}}
\newcommand*{\oy}{\overline{y}}
\newcommand*{\oz}{\overline{z}}

\newcommand*{\oxi}{\overline{\xi}}

\hypersetup{
	colorlinks=true,
	linkcolor=magenta,
	filecolor=black,      
	urlcolor=blue,
	citecolor=cyan,
}

\numberwithin{equation}{section}

\newtheoremstyle{note}
    {8pt}
    {6pt}
    {\itshape}
    {11.75pt}
    {\bfseries}
    {.}
    {.5em}
    {}
\theoremstyle{note}
\newtheorem*{thm*}{Theorem}
\newtheorem{thm}{Theorem}[section]
\newtheorem{lem}[thm]{Lemma}
\newtheorem*{prop*}{Proposition}

\newtheoremstyle{note2}
    {8pt}
    {1pt}
    {}
    {11.75pt}
    {\bfseries}
    {.}
    {.5em}
    {}
\theoremstyle{note2}
\newtheorem{defn}{Definition}[section]

\newtheoremstyle{note3}
    {8pt}
    {1pt}
    {}
    {11.75pt}
    {\itshape}
    {.}
    {.5em}
    {}
\theoremstyle{note3}
\newtheorem{rem}{Remark}[section]

\DeclareMathAlphabet{\pazocal}{OMS}{zplm}{m}{n}

\begin{document}
	\title[Removable singularities for Lipschitz fractional caloric functions]{Removable singularities for Lipschitz fractional caloric functions in time varying domains}
	\author{Joan Hernández}\thanks{ORCID: 0000-0002-2207-5981. The author has been supported by PID2020-114167GB-I00 (Mineco, Spain).}
	
	\begin{abstract}
		In this paper we study removable singularities for regular $(1,\frac{1}{2s})$-Lipschitz solutions of the $s$-fractional heat equation for $1/2<s<1$. To do so, we define a Lipschitz fractional caloric capacity and study its  critical dimension and the $L^2$-boundedness of a pair of singular integral operators, whose kernels will be the gradient of the fundamental solution of the fractional heat equation and its conjugate.
		\bigskip
		
		\noindent\textbf{AMS 2020 Mathematics Subject Classification:}  42B20 (primary); 28A12 (secondary).
		
		\medskip
		
		\noindent \textbf{Keywords:} Removable singularities, fractional heat equation, singular integrals, $L^2$-boundedness.
	\end{abstract}

    	\maketitle

	\section{Introduction}
	\label{sec1}
    \noindent
	In the article \cite{MPrT}, Mateu, Prat and Tolsa introduce the notion of Lipschitz caloric removability. To define it, objects from parabolic theory in time varying domains are required, such as the parabolic distance, the parabolic BMO space, or Lipschitz parabolic functions (see \cite{Ho}, \cite{HoL}, \cite{NyS}). Inspired by \cite{MPr}, we extend these ideas to the fractional caloric setting, giving rise to Lipschitz $s$-parabolic functions, as well as a characterization of Lipschitz $s$-caloric removability using a certain capacity. Intuitively, removable sets can be thought of as those that do not affect the search for solutions of the $s$-heat equation with the previous regularity assumptions. That is, any $s$-parabolic Lipschitz function solving the $s$-heat equation outside the set actually verifies the same equation throughout the entire domain, including the set itself.
    
    We introduce basic notation and terminology. The ambient space will be $\mathbb{R}^{n+1}$ with a generic point denoted by $\ox:=(x,t)\in\mathbb{R}^n\times\mathbb{R}$. For each  $s\in(0,1]$, the $s$-heat operator is defined as
    \begin{equation*}
        \Theta^s:=(-\Delta)^s+\partial_t,
    \end{equation*}
    and we name its associated equation the $\Theta^s$-equation. If $s<1$, $(-\Delta)^s:=(-\Delta_x)^{s}$ is the pseudo-differential operator known as the \textit{$s$-Laplacian} with respect to spatial variables. It can be defined through its Fourier transform,
	\begin{equation*}
		\widehat{(-\Delta_x)^{s}f}(\xi,t)=|\xi|^{2s}\widehat{f}(\xi,t),
	\end{equation*}
	or by its integral representation
	\begin{align*}
		(-\Delta_x)^{s} f(x,t)& = c_{n,s}\text{p.v.}\int_{\mathbb{R}^n}\frac{f(x,t)-f(y,t)}{|x-y|^{n+2s}}\dd y  \\
		& = c_{n,s}' \int_{\mathbb{R}^n}\frac{f(x+y,t)-2f(x,t)+f(x-y,t)}{|y|^{n+2s}}\dd y.
	\end{align*}
	The reader may find more details about the properties of such non-local operator in \cite[\textsection{3}]{DPV} or \cite{St}.
    
    The fundamental solution $P_s(x,t)$ of the $\Theta^s$-equation is given by the inverse spatial Fourier transform of $e^{-4\pi^2 t|\xi|^{2s}}$ for $t>0$, and it equals $0$ if $t\leq 0$. If $s=1$, the function $P_1$ is nothing but the classical \textit{heat kernel},
\begin{equation*}
    W(\ox):=P_1(\ox)=c_nt^{-n/2}\phi_{n,1}(|x|t^{-1/2}) \chi_{t>0},
\end{equation*}
with $\phi_{n,1}(\rho):=e^{-\rho^2/4}$. Blumenthal and Getoor \cite[Theorem 2.1]{BG} proved that for $s<1$,
\begin{equation*}
    P_s(\ox)=c_{n,s}''\,t^{-\frac{n}{2s}}\phi_{n,s}\big( |x|t^{-\frac{1}{2s}} \big)\chi_{t>0}.
\end{equation*}
Here, $\phi_{n,s}$ is a smooth function, radially decreasing and satisfying,
\begin{equation*}
    \phi_{n,s}(\rho)\approx_{n,s} \big(1+\rho^2\big)^{-(n+2s)/2}, \qquad \text{if \ $0<s<1$},
\end{equation*}
the latter being an equality if $s=1/2$ \cite{Va}. Hence,
\begin{equation*}
    P_s(\ox)\approx_{n,s} \frac{t}{|\ox|_{p_s}^{n+2s}}\chi_{t>0}.
\end{equation*}
It follows from the construction that
\begin{equation*}
     \widehat{\phi_{n,s}\big(|\,\cdot\,|\big)}(\xi)=e^{-4\pi^2 |\xi|^{2s}}.
\end{equation*}

    For each $0<s\leq 1$, we define the $s$\textit{-parabolic distance} between two points $\ox:=(x,t),\oy:=(y,\tau)$ in $\mathbb{R}^{n+1}$ as
\begin{align*}
    |\ox-\oy|_{p_s}=\text{dist}_{p_s}((x,t),(y,\tau))&:=\max\big\{ |x-y|,|t-\tau|^{\frac{1}{2s}} \big\}\\
    &\;\approx_{n,s} (|x-y|^2+|t-\tau|^{1/s})^{1/2}.
\end{align*}
From the latter, the notions of $s$\textit{-parabolic cube} and $s$\textit{-parabolic ball} emerge naturally. We denote $B(\ox,r)$ as the $s$-parabolic ball centered at $\ox=(x,t)$ with radius $r$, and $Q$ as an $s$-parabolic cube of side length $\ell$. Observe that $Q$ is a set of the form
\begin{equation*}
     I_1\times\cdots\times I_n\times I_{n+1},
\end{equation*}
where $I_1,\ldots, I_n$ are intervals of length $\ell$, while $I_{n+1}$ is another interval of length $\ell^{2s}$. Write $\ell(Q)$ to refer to the particular side length of $Q$. Observe that $B(\ox,r)$ can be also decomposed as a cartesian product of the form $B_1\times I$, where $B_1\subset \mathbb{R}^n$ is the Euclidean ball of radius $r$ centered at $x$ and $I\subset \mathbb{R}$ is a real interval of length $(2r)^{2s}$ centered at $t$.

In an $s$-parabolic setting, we say that $\delta_\lambda$ is an $s$\textit{-parabolic dilation} of factor $\lambda>0$ if 
\begin{equation*}
    \delta_\lambda(x,t)=\big(\lambda x, \lambda^{2s} t\big
    ).
\end{equation*}
Since we will always work with the $s$-parabolic distance, for a given $s$-parabolic cube $Q\subset \mathbb{R}^{n+1}$ we will simply write $\lambda Q$ to denote $\delta_{\lambda}(Q)$. That is, $\lambda Q$ will be the $s$-parabolic cube of side length $\lambda \ell(Q)$ concentric with $Q$.

As the reader may suspect, the notion of $s$\textit{-parabolic BMO space}, $\text{BMO}_{p_s}$, refers to the space of functions (strictly, equivalence classes) that present bounded mean oscillation over all $s$-parabolic cubes, instead of the usual Euclidean ones. We will write $\|\cdot\|_{\ast,p_s}$ to denote the $s$-parabolic BMO norm, i.e.
\begin{equation*}
    \|f\|_{\ast,p_s}:=\sup_{Q}\frac{1}{|Q|}\int_Q|f-f_Q|\dd\pazocal{L}^{n+1},
\end{equation*}
the supremum taken among $s$-parabolic cubes in $\mathbb{R}^{n+1}$. Here $\dd\pazocal{L}^{n+1}$ stands for the Lebesgue measure in $\mathbb{R}^{n+1}$ and $f_Q$ is the mean of $f$ in $Q$ with respect to $\dd\pazocal{L}^{n+1}$. We will also need the following fractional time derivative for $s\in(1/2,1]$,
\begin{equation*}
    \partial_t^{\frac{1}{2s}}f(x,t):=\int_\mathbb{R}\frac{f(x,\tau)-f(x,t)}{|\tau-t|^{1+\frac{1}{2s}}}\dd\tau.
\end{equation*}

For $s=1$ (the classical caloric setting) a function $f$ is said to be $(1,1/2)$-Lipschitz regular if it is such that
\begin{equation}
	\label{eq1.1}
	\|\nabla_x f\|_{L^{\infty}(\mathbb{R}^{n+1})}<\infty, \hspace{1cm} \|\partial_t^{1/2} f\|_{\ast, p_1}<\infty.
\end{equation}
As shown by Hofmann and Lewis \cite[Lemma 1]{Ho}, \cite[Thm. 7.4]{HoL}, these functions satisfy
\begin{equation*}
	\|f\|_{\text{Lip}_{1/2},t}:=\sup_{\substack{x\in \mathbb{R}^n\\ t,u\in \mathbb{R}, t\neq u}} \frac{|f(x,t)-f(x,u)|}{|t-u|^{1/2}}\lesssim \|\nabla_x f\|_{L^\infty(\mathbb{R}^{n+1})}+\|\partial_t^{1/2}f\|_{\ast,p_1}.
\end{equation*}
Thus a $(1,1/2)-$Lipschitz function is Lipschitz in the spatial variables and 1/2-Lipschitz in time. In \cite{MPrT} the authors introduce the so-called Lipschitz caloric capacity to characterize removable sets for solutions of the heat equation satisfying the above $(1,1/2)-$-Lipschitz condition

The main goal of this paper is to extend the previous notions and results to the fractional caloric setting. Namely, we will introduce a corresponding capacity $\Gamma_{\Theta^s}$ for $(1,\frac{1}{2s})$-Lipschitz solutions $f$ of the $\Theta^s$-equation, $1/2<s<1$, i.e. those satisfying
\begin{equation*}
	\|\nabla_x f\|_{L^\infty(\mathbb{R}^{n+1})}\lesssim 1\qquad \text{and} \qquad \|f\|_{\text{Lip}_{\frac{1}{2s},t}} :=\sup_{\substack{x\in \mathbb{R}^n\\ t,u\in \mathbb{R}, t\neq u}} \frac{|f(x,t)-f(x,u)|}{|t-u|^{\frac{1}{2s}}} \lesssim 1.
\end{equation*}
We begin our analysis in \textsection \ref{sec2} by noting that the results of Hofmann and Lewis \cite{HoL} can be adapted to the fractional caloric case:
\begin{thm*}
	Let $s\in(1/2,1]$ and $f:\mathbb{R}^{n+1}\to \mathbb{R}$ be such that
    \begin{equation}
    \label{eq1.1}
        \|\nabla_x f\|_{L^\infty(\mathbb{R}^{n+1})}\lesssim 1, \quad \|\partial_t^{\frac{1}{2s}}f\|_{\ast,p_s}\lesssim 1.
    \end{equation}
    Then, $f$ is $(1,\frac{1}{2s})$-Lipschitz. 
\end{thm*}

The proof of the above result follows the arguments of \cite{HoL} and adapts them to the fractional parabolic setting. For instance, one needs to work with $s$-parabolic Riesz transforms instead of the usual ones from the Euclidean case. Moreover, we need to follow the dependence on $s$ when establishing the above estimates and, in fact, in Remark \ref{rem2.1} we observe that the arguments break down when $s=1/2$, which is consistent with the theory.

This theorem allows us to define the capacity $\Gamma_{\Theta^s}$ in analogy with the classical case \cite{MPrT}. For instance, given a compact set $E\subset\mathbb{R}^{n+1}$, $\Gamma_{\Theta^s}$ is 
$$\Gamma_{\Theta^s}(E):=\sup\{|\langle \Theta^s f,1\rangle|\textbf{}\},$$
the supremum taken over all $(1,\frac{1}{2s})-$Lipschitz regular functions $f:\mathbb{R}^{n+1}\to\mathbb{R}$ satisfying the $s$-heat equation on $\mathbb{R}^{n+1}\setminus E$ and with the norms in \eqref{eq1.1} smaller or equal than one.

In \textsection \ref{sec3.2} we localize the potentials associated with the kernels $\nabla_xP_s$ and $\partial_t^{\frac{1}{2s}}P_s$. Although our methods of proof and statement of results will be analogous to those of \cite{MPrT} for the classical heat case, we will need to adapt many techniques and generalize some of the methods used in the previous reference. Indeed, we need to take into account that the operator $\Theta^s$ for $s<1$ is no longer local, so many properties will no longer hold. For example, if $s=1$ one can exploit the useful relation
\begin{equation*}
    \Delta (fg) = g \Delta f + f\Delta g + 2 \nabla f \cdot \nabla g.
\end{equation*}
However, if $s<1$ one has to work with expressions of the form
\begin{equation*}
		(-\Delta)^s(fg) = f(-\Delta)^sg+g(-\Delta)^sf-I_s(f,g),
	\end{equation*}
where $I_s$ is a crossed-terms integral defined as
\begin{equation*}
    I_s(f,g)\simeq \int_{\mathbb{R}^n}\frac{(f(x)-f(y))(g(x)-g(y))}{|x-y|^{n+2s}}\dd y.
\end{equation*}

In \textsection\ref{sec3.3} we use the localization of potentials to prove the equivalence between the removability of compact sets for $(1,\frac{1}{2s})$-Lipschitz solutions of the $\Theta^s$-equation and the nullity of $\Gamma_{\Theta^s}$. More precisely, a compact set $E\subset \mathbb{R}^{n+1}$ is said to be Lipschitz $s$-caloric removable if for any open subset $\Omega\subset \mathbb{R}^{n+1}$, any function $f:\mathbb{R}^{n+1}\to \mathbb{R}$ with $\|\nabla_x f\|_{L^\infty(\mathbb{R}^{n+1})}<\infty $ and $\|\partial_t^{\frac{1}{2s}}f\|_{\ast,p_s}<\infty$ satisfying the $\Theta^s$-equation in $\Omega\setminus{E}$, also satisfies the previous equation in $\Omega$. The main theorem of this article reads as follows:
\begin{thm*}
	Let $s\in(1/2,1]$. A compact set $E\subset \mathbb{R}^{n+1}$ is removable for Lipschitz $s$-caloric functions if and only if $\Gamma_{\Theta^s}(E)=0$.
\end{thm*}

Moreover, we establish that in $\mathbb{R}^{n+1}$ the critical $s$-parabolic Hausdorff dimension of the previous capacity is $n+1$ and provide an example, for each $s\in(1/2,1)$, of a subset with positive $\pazocal{H}_{p_s}^{n+1}$-measure and null capacity. This already shows that the capacity $\Gamma_{\Theta^s}$ has to be fundamentally different to those studied by the author in \cite{HeMPr}, since it cannot be comparable to the $s$-parabolic Hausdorff content $\pazocal{H}_{\infty, p_s}^{n+1}$. Such set will be a generalization of the Cantor set defined in \cite[\textsection 6]{MPrT} whose construction will heavily depend on the parameter $s$. In fact, as $s$ approaches the value $1/2$, our Cantor set will become progressively more and more dense (even for the first iteration of its construction) in the unit cube. 

As the reader may have noticed, our fractional generalization only considers $s\in(1/2,1]$. In fact, most of the arguments we develop below are inspired by those in \cite{Ho, HoL, MPrT}, and they break down if one considers the limit $s\to 1/2$ (in the sense that most implicit constants blow up). Hence, the study of Lipschitz $s$-caloric functions and their associated capacity remains still an open problem in the regime $s\in(0,1/2]$. The author conjectures that in the particular case $s=1/2$, where the $s$-parabolic gradient simplifies to the usual gradient and the $s$-parabolic distance is nothing but the Euclidean one, the associated capacity of a compact set should be comparable to the Lebesgue measure of the ambient space restricted to that set. The results obtained in \textsection \ref{subsec3.3.1} regarding Cantor sets together with the work of Uy in \cite{U} for analytic capacity suggest this comparability.

Finally, in \textsection \ref{sec5}, motivated by a question posed by X. Tolsa, we revisit the classical parabolic case $s=1$ and define a new capacity, different to that of \cite{MPrT}, now working with solutions of the heat equation satisfying
\begin{equation*}
	\|(-\Delta_x)^{1/2} f\|_{L^\infty(\mathbb{R})}\lesssim 1\qquad \text{and} \qquad \|\partial_t^{1/2}f\|_{\ast,p_1}\lesssim 1.
\end{equation*}
We call such capacity $\gamma_{\Theta}^{1/2}$. We prove that it shares the same critical dimension with the Lipschitz caloric capacity but observe that, at least in the plane, we have
\begin{thm*}
    The capacities $\Gamma_{\Theta}$ and $\gamma^{1/2}_{\Theta}$ are not comparable.
\end{thm*}

\textit{About the notation used in the sequel}: Constants appearing in the sequel may depend on the dimension of the ambient space and the  parameter $s$, and their value may change at different occurrences. They will frequently be denoted by the letters $c$ or $C$. The notation $A\lesssim B$ means that there exists $C$, such that $A\leq CB$. Moreover, $A\approx B$ is equivalent to $A\lesssim B \lesssim A$, while $A \simeq B$ will mean $A= CB$. If the reader finds expressions of the form $\lesssim_{\beta}$ or $\approx_\beta$, for example, this indicates that the implicit constants depend on $n,s$ and $\beta$.
    
	Since Laplacian operators (fractional or not) will frequently appear in our discussion and will always be taken with respect to spatial variables, we will adopt the notation:
	\begin{equation*}
		(-\Delta)^s:=(-\Delta_x)^s, \qquad s\in(0,1], \quad \text{and we convey $(-\Delta)^0:=\text{Id}$}.
	\end{equation*}
	We will also write $\|\cdot\|_\infty:=\|\cdot\|_{L^\infty(\mathbb{R}^{n+1})}$.
    
    In \textsection\ref{sec2}, \textsection\ref{sec3.2} and \textsection\ref{sec3.3} we fix $s\in(1/2,1)$ . The forthcoming results for the $1$-parabolic case are already covered in \cite{MPrT}.
    
\section{The \mathinhead{\Gamma_{\Theta^s}}{} capacity. Generalizing the \mathinhead{(1,\frac{1}{2s})}{}-Lipschitz condition}
\label{sec2}
We begin by introducing an equivalent definition of the $s$-parabolic norm, which is presented in \cite{Ho}. For $\ox:=(x,t)\in \mathbb{R}^{n+1}$, the quantity $\|\ox\|_{p_s}$ is defined to be the only positive solution of
\begin{equation}
	\label{eq3.0.1}
	1 = \sum_{j=1}^n \frac{x_j^2}{\|\ox\|_{p_s}^2}+\frac{t^2}{\|\ox\|_{p_s}^{4s}}.
\end{equation}
If $s=1$, one recovers the proper $1$-parabolic norm already defined in \cite{Ho}, given by the explicit expression
\begin{equation*}
	\|\ox\|_{p_1}^2=\frac{1}{2}\Big( |x|^2+\sqrt{|x|^4+4t^2} \Big).
\end{equation*}
This quantity is comparable to the parabolic norm presented in the previous section $|\ox|_{p_1}^2:=|x|^2+|t|$. In \cite{MPr}, Mateu and Prat study the case $s<1$ and introduce the following quantity comparable to $\|\ox\|_{p_s}^2$,
\begin{equation*}
	|\ox|_{p_s}^2:=|x|^2+|t|^{1/s},
\end{equation*}
that is precisely the expression we will use for the $s$-parabolic norm in this article. However, the choice of $\|\cdot\|_{p_s}$ instead of $|\cdot|_{p_s}$ presents some advantages regarding Fourier representation formulae of certain operators, as we will see in the sequel.

As in \cite{Ho} we define the $s$-parabolic fractional integral operator of order $1$ as
\begin{equation*}
    \widehat{I_{p_s} f} (\oxi) := \|\oxi\|_{p_s}^{-1}\widehat{f}(\oxi),
\end{equation*}
and its inverse
\begin{equation*}
	\widehat{D_{p_s} f} (\oxi) := \|\oxi\|_{p_s}\widehat{f}(\oxi).
\end{equation*}
Then, multiplying both sides of \eqref{eq3.0.1} by $\|x\|_{p_s}$ and taking the Fourier transform we obtain the following identity between operators
\begin{equation}
	\label{eq3.1.2}
	D_{p_s}\simeq \sum_{j=1}^n R_{j,s}\partial_j + R_{n+1,s}D_{n+1,s},
\end{equation}
where
\begin{equation*}
	\widehat{R_{j,s}} := \frac{\xi_j}{\|\oxi\|_{p_s}}, \qquad j=1,\ldots,n,
\end{equation*}
are the $s$-parabolic Riesz transforms. As observed in the comments that follow \cite[Equation (2.10)]{HoL}, operators $R_{j,s}$ satisfy analogous properties to those of usual Riesz transforms. Indeed, observe that their symbols are antisymmetric and $s$-parabolically homogeneous of degree $0$. Thus, $R_{j,s}$ are defined via convolution (in a principal value sense) against odd kernels, $s$-parabolically homogeneous of degree $-n-2s$, which are bounded in $L^p$ and $\text{BMO}_{p_s}$ (see \cite[Remark 1.3]{Pe}). We have also set $\oxi:=(\xi,\tau)$ and defined
\begin{equation*}
	\widehat{R_{n+1,s}}(\oxi):=\frac{\tau}{\|\oxi\|_{p_s}^{2s}}, \qquad \widehat{D_{n+1,s}}(\oxi):=\frac{\tau}{\|\oxi\|_{p_s}^{2s-1}}.
\end{equation*}
Observe that if $s=1/2$, the operator $D_{n+1,s}$ is nothing but a temporal derivative. The result \cite[Theorem 7.4]{HoL} establishes the comparability between  $\|D_{n+1,1}f\|_{\ast,p_1}$ and $\|\partial_t^{1/2}f\|_{\ast,p_1}$ in the $1$-parabolic case under the assumption $\|\nabla_x f\|_{\infty}\lesssim 1$. For our purposes, we will only be interested in the estimate
\begin{equation}
	\label{new_eq3.1.3}
	\|D_{n+1,s}f\|_{\ast,p_s}\lesssim \|\partial_t^{\frac{1}{2s}}f\|_{\ast,p_s}+\|\nabla_x f\|_{\infty}.
\end{equation}
To prove such a bound we simply follow the proof of \cite[Theorem 7.4]{HoL} and make dimensional adjustments. First, we observe that
\begin{equation*}
	\widehat{D_{n+1,s}}f(\oxi) = \Big[ |\tau|^{\frac{1}{2s}}m(\oxi) \Big]\widehat{f}, \qquad \text{where }\quad m(\oxi):=\frac{\tau}{|\tau|^{\frac{1}{2s}}\|\oxi\|_{p_s}^{2s-1}}.
\end{equation*}
Now, since $m$ is not smooth in $\mathbb{R}^{n+1}\setminus{\{0\}}$, let us consider an even function $\phi\in \pazocal{C}^{\infty}(\mathbb{R})$, which equals 1 on $(2/K,\infty)$ and is supported on $(-\infty,-1/K)\cup(1/K,\infty)$, for some $K\geq 2$. We also choose $\phi$ so that $\|\partial^l\phi \|_{\infty}\lesssim K^l$, for $0\leq l \leq n+5$. Let us introduce the following auxiliary functions,
\begin{align*}
	m^+(\oxi):=m(\oxi)\phi\bigg( \frac{\tau}{|\xi|^{2s}} \bigg), \qquad m^{++}(\oxi)&:=\frac{|\tau|^{\frac{1}{2s}}\|\oxi\|_{p_s}}{|\xi|^2}m(\oxi)\big(1-\phi \big)\bigg( \frac{\tau}{|\xi|^{2s}} \bigg),\\
	m_j^{++}(\oxi)&:=\frac{\xi_j}{\|\oxi\|_{p_s}}m^{++}(\oxi).
\end{align*}
This way, we rewrite $\widehat{D_{n+1,s}f}$ as follows:
\begin{align}
	\widehat{D_{n+1,s}f}(\oxi) &= \bigg[ m^+(\oxi)|\tau|^{\frac{1}{2s}} + m(\oxi)\big( 1-\phi \big)\bigg( \frac{\tau}{|\xi|^{2s}} \bigg)|\tau|^{\frac{1}{2s}} \bigg]\widehat{f}(\oxi)\nonumber \\
	&=\bigg[ m^+(\oxi)|\tau|^{\frac{1}{2s}} + \sum_{j=1}^n \frac{\xi_j}{\|\oxi\|_{p_s}} m^{++}(\oxi) \xi_j \bigg]\widehat{f}(\oxi)\nonumber \\
	&=\bigg[ m^+(\oxi)|\tau|^{\frac{1}{2s}} + \sum_{j=1}^n m^{++}_j(\oxi) \xi_j \bigg]\widehat{f}(\oxi).
	\label{eq3.1.3}
\end{align}
Regarding $m^+$, observe that it is an odd multiplier, smooth in $\mathbb{R}^{n+1}\setminus{\{0\}}$ and $s$-parabolically homogeneous of degree 0. Therefore, by \cite[Proposition 2.4.7]{Gr} (which admits a direct adaptation to the $s$-parabolic case), $m^+$ is associated to a convolution kernel $L^+$, odd, $s$-parabolically homogeneous of degree $-n-2s$ and bounded on $L^p$ and $\text{BMO}_{p_s}$ (see again \cite[Remark 1.3]{Pe}).

Let us turn to $m_j^{++}$, which in particular lacks the smoothness properties of $m^+$. Our first goal will be to prove that there exists $L_j^{++}$ convolution kernel associated to $m_j^{++}$ and that is bounded from $L^\infty$ to $\text{BMO}_{p_s}$. We begin by noticing that
\begin{equation*}
	\text{supp}(m_j^{++})\subset \bigg\{ (\xi,\tau)\,:\, 0\leq |\tau| \leq \frac{2|\xi|^{2s}}{K} \bigg\},
\end{equation*}
meaning that in the support of $m_j^{++}$ we have $\|\xi\|_{p_s}\approx |\xi|$. Moreover,
\begin{equation*}
	|m_j^{++}(\oxi)|=\frac{|\tau|^{\frac{1}{2s}}|\xi_j|}{|\xi|^2}\big(1-\phi \big)\bigg( \frac{\tau}{|\xi|^{2s}} \bigg)|m(\oxi)|\lesssim |m(\oxi)|\leq \frac{|\tau|^{1-\frac{1}{2s}}}{\|\oxi\|_{p_s}^{2s-1}}\lesssim 1.
\end{equation*}
In fact, more generally we have for $\alpha\in\{0,1,2,\ldots\}$ and $\beta\in\{0,1,2\ldots\}^n$,
\begin{align}
	\label{eq3.1.4}
	|\partial_{\tau}^{\alpha} \partial_{\xi}m_j^{++}(\oxi)|&\lesssim |\tau|^{1-\frac{1}{2s}-\alpha}\|\oxi\|_{p_s}^{-2s+1-|\beta|},\\
	\label{eq3.1.5} |\partial_{\tau}^{\alpha} \partial_{\xi}(\xi_jm_j^{++})(\oxi)|&\lesssim |\tau|^{1-\frac{1}{2s}-\alpha}\|\oxi\|_{p_s}^{-2s+2-|\beta|},\\
	\label{eq3.1.6} |\partial_{\tau}^{\alpha} \partial_{\xi}(\tau m_j^{++})(\oxi)|&\lesssim |\tau|^{2-\frac{1}{2s}-\alpha}\|\oxi\|_{p_s}^{-2s+1-|\beta|}.
\end{align}
We will use the above inequalities to justify the existence of $L_j^{++}$. To do so, let $\{g_i\}$ be a smooth partition of unity of $(0,\infty)$ with $g_i\equiv 1$ on $(2^{-i},2^{-i+1})$ and $\text{supp}(g_i)\subset (2^{-i-1},2^{-i+2})$, for $i\in \mathbb{Z}$. Let $\{L_{j,i}^{++}\}$ be the kernels corresponding to
\begin{equation*}
	m_{j,i}^{++}(\oxi):=m_{j}^{++}(\oxi)\cdot g_i(\|\oxi\|_{p_s}),
\end{equation*}
so that
\begin{equation*}
	\text{supp}(m_{j,i}^{++})\subset \bigg\{ (\xi,\tau)\,:\, 0\leq \frac{\tau}{|\xi|^{2s}} \leq \frac{2}{K}, \; \frac{1}{2^{i+1}}\leq \|\oxi\|_{p_s}\leq \frac{1}{2^{i-2}} \bigg\}.
\end{equation*}
We observe that
\begin{equation*}
	\|L_{j,i}^{++}\|_\infty \lesssim \|m_{j,i}^{++}\|_1\lesssim 2^{-i(n+2s)}, \qquad \text{since }\quad \|m_{j}^{++}\|_{\infty}\lesssim 1.
\end{equation*}
Similarly, for any $\beta\in\{0,1,2,\ldots\}^n$ with $|\beta|=n+3$ we have by \eqref{eq3.1.4},
\begin{equation*}
	|\xi^\beta L_{j,i}^{++}(\xi,\tau)|\lesssim \|\partial_\xi^\beta m_{j,i}^{++}\|_1\lesssim 2^{i(3-2s)}.
\end{equation*}
Moreover, for any $\beta\in\{0,1,2,\ldots\}^n$ with $|\beta|=n$,
\begin{equation*}
	\big\rvert |\tau|^{\frac{1}{4s}}\tau\xi^\beta L_{j,i}^{++}(\xi,\tau) \big\rvert \lesssim \|\partial_\tau^{\frac{1}{4s}}(\partial_\tau \partial_\xi^\beta m_{j,i}^{++})\|_1.
\end{equation*}
Using again \eqref{eq3.1.4} and the definition of $\partial_\tau^{\frac{1}{4s}}$ we get, setting $\sigma:= \partial_\tau \partial_\xi^\beta m_{j,i}^{++}$,
\begin{align*}
	\bigg\rvert \int_{\mathbb{R}}\frac{\sigma(\xi,\tau)-\sigma(\xi,r)}{|t-r|^{1+\frac{1}{4s}}} \dd r \bigg\rvert 
	\begin{dcases}
		= 0, & \text{if }\, |\xi|\geq C2^{-i}, \\
		\lesssim \frac{2^{in}}{|\tau|^{1+\frac{1}{4s}}}, & \text{if }\, |\tau| > \frac{4|\xi|^{2s}}{K}, \\
		\lesssim \frac{2^{i(n+2s-1)}}{|\tau|^{\frac{3}{4s}}}, & \text{if }\, |\tau|\leq \frac{8|\xi|^{2s}}{K}\leq C'2^{-2si}, \\
	\end{dcases}
\end{align*}
where $C,C'$ are constants large enough, depending on $n$ and $s$. Then,
\begin{equation*}
	\|\partial_\tau^{\frac{1}{4s}}(\partial_\tau \partial_\xi^\beta m_{j,i}^{++})\|_1 \lesssim 2^{i/2}.
\end{equation*}
Therefore, we have obtained
\begin{equation*}
	|L_{j,i}^{++}(\xi,\tau)|\lesssim \min\big\{ 2^{-i(n+2s)}, 2^{i(3-2s)}|\xi|^{-(n+3)}, 2^{i/2}|\tau|^{-1-\frac{1}{4s}}|\xi|^{-n} \big\}.
\end{equation*}
Proceeding analogously using \eqref{eq3.1.5} and \eqref{eq3.1.6}, the following bounds follow
\begin{align*}
	|\nabla_\xi L_{j,i}^{++}(\xi,\tau)|&\lesssim \min\big\{ 2^{-i(n+2s+1)}, 2^{i(3-2s)}|\xi|^{-(n+4)}, 2^{i/2}|\tau|^{-1-\frac{1}{4s}}|\xi|^{-(n+1)} \big\},\\
	|\partial_\tau L_{j,i}^{++}(\xi,\tau)|&\lesssim \min\big\{ 2^{-i(n+4s)}, 2^{i(5-4s)}|\xi|^{-(n+5)}, 2^{i/2}|\tau|^{-1-\frac{1}{4s}}|\xi|^{-(n+2)} \big\}.
\end{align*}
Now let $L_j^{++}:=\sum_i L_{j,i}^{++}$ whenever the sum converges absolutely. Let us distinguish two cases: if $|\xi|^{2s+1/2}\gtrsim |\tau|^{1+\frac{1},{4s}}$ we get
\begin{align*}
	|L_{j}^{++}(\xi,\tau)| &\leq \sum_i |L_{j,i}^{++}(\xi,\tau)|\\
	&\lesssim \sum_{\{i:\, |\xi|\leq 2^i\}} 2^{-i(n+2s)}+\sum_{\{i:\, |\xi|>2^i\}} 2^{i(3-2s)}|\xi|^{-(n+3)}\lesssim |\xi|^{-(n+2s)}.
\end{align*}
If $|\xi|^{2s+1/2}\lesssim |\tau|^{1+\frac{1}{4s}}$ and $\lambda>0$, we have
\begin{align*}
	|L_{j}^{++}(\xi,\tau)| &\leq \sum_i |L_{j,i}^{++}(\xi,\tau)|\\
	&\leq \sum_{\{i:\, 2^{-i}\leq \lambda\}} C_12^{-i(n+2s)}+C_2|\tau|^{-1-\frac{1}{4s}}|\xi|^{-n}\sum_{\{i:\, 2^{-i}>\lambda\}} 2^{i/2}\\
	&\leq C_1\lambda^{n+2s}+C_2|\tau|^{-1-\frac{1}{4s}}|\xi|^{-n}\lambda^{-1/2}.
\end{align*}
This last expression is minimized for
\begin{equation*}
	\lambda \simeq \Big( |\tau|^{-1-\frac{1}{4s}}|\xi|^{-n} \Big)^{\frac{1}{n+2s+1/2}},
\end{equation*}
and therefore
\begin{equation*}
	|L_{j}^{++}(\xi,\tau)|\lesssim \Big( |\tau|^{-1-\frac{1}{4s}}|\xi|^{-n} \Big)^{\frac{n+2s}{n+2s+1/2}},
\end{equation*}
that, under condition $|\xi|^{2s+1/2}\lesssim |\tau|^{1+\frac{1}{4s}}$, is smaller than $|\xi|^{-(n+2s)}$. So we have obtained that $L_j^{++}$ is well-defined in $\mathbb{R}^{n+1}\setminus{\{0\}}$ and satisfies
\begin{equation*}
	|L_{j}^{++}(\xi,\tau)|\lesssim \min\Big\{ |\xi|^{-(n+2s)}, \Big( |\tau|^{-1-\frac{1}{4s}}|\xi|^{-n} \Big)^{\frac{n+2s}{n+2s+1/2}} \Big\}.
\end{equation*}
Again, following an analogous procedure one gets
\begin{align*}
	|\nabla_\xi L_{j}^{++}(\xi,\tau)|&\lesssim \min\Big\{ |\xi|^{-(n+2s+1)}, \Big( |\tau|^{-1-\frac{1}{4s}}|\xi|^{-n-1} \Big)^{\frac{n+2s}{n+2s+1/2}} \Big\},\\
	|\partial_\tau L_{j}^{++}(\xi,\tau)|&\lesssim \min\Big\{ |\xi|^{-(n+4s)}, \Big( |\tau|^{-1-\frac{1}{4s}}|\xi|^{-n-2} \Big)^{\frac{n+2s}{n+2s+1/2}} \Big\}.
\end{align*}
Let us also notice that for any dimension $n$ and, in fact, any $s\in (0,1)$,
\begin{equation*}
	\bigg( 1+\frac{1}{4s} \bigg)\bigg( \frac{n+2s}{n+2s+1/2} \bigg) \geq \frac{17}{16}>1.
\end{equation*}
From this point on, we are able to follow the arguments of \cite[p. 407]{HoL} to deduce that the principal value convolution operator associated to $L_j^{++}$ exists and maps $L^\infty$ to $\text{BMO}_{p_s}$ (the arguments are those already given in \cite{Pe}). So returning to \eqref{eq3.1.3} we have
\begin{equation*}
	D_{n+1,s}f(\ox) = C_1(L^+\ast \partial_t^{\frac{1}{2s}}f)(\ox)+C_2\sum_{j=1}^n(L_j^{++}\ast \partial_j f)(\ox),
\end{equation*}
and therefore
\begin{equation*}
	\|D_{n+1,s}f\|_{\ast,p_s}\lesssim \|\partial_t^{\frac{1}{2s}}f\|_{\ast,p_s}+\|\nabla_x f\|_\infty,
\end{equation*}
that is what we wanted to prove. Let us also mention that the roll of the parameter $K\geq 2$ introduced in the above arguments becomes clear if one follows its dependence in the different inequalities previously established. Doing so, and adapting the arguments of \cite[Theorem 7.4]{HoL}, one is able to prove a reverse inequality of the form $\|\partial_t^{\frac{1}{2s}}f\|_{\ast,p_s}\lesssim \|D_{n+1,s}f\|_{\ast,p_s} + \|\nabla_x f\|_\infty$, although we have not presented the details since it is not necessary in our context.

In any case, now is just a matter of applying \cite[Lemma 1]{Ho} to a function $f$ satisfying $\|\nabla_x f\|_\infty \lesssim 1$ and $\|\partial_t^{\frac{1}{2s}}f\|_{\ast,p_s}\lesssim 1$. Then, returning to \eqref{eq3.1.2},
\begin{align*}
	\|D_{p_s}f\|_{\ast,p_s}&\lesssim \sum_{j=1}^n \|R_{j,s}\partial_j f\|_{\ast,p_s} + \|R_{n+1,s}D_{n+1,s} f \|_{\ast,p_s}\\
	&\lesssim \|\nabla_x f\|_\infty +\|D_{n+1,s}f\|_{\ast,p_s} \lesssim \|\nabla_x f\|_{\infty} + \|\partial_t^{\frac{1}{2s}}f\|_{\ast,p_s}\lesssim 1,
\end{align*}
where this last inequality is due to \eqref{new_eq3.1.3}. Therefore, by \cite[Lemma 1]{Ho}, which also holds in the fractional parabolic context (since we also have $H^1-\text{BMO}_{p_s}$ duality (see \cite[\textsection 2, Theorem B]{CoW}, for example) and because $s$-parabolic Riesz kernels enjoy the expected regularity properties a mentioned above), we finally get the desired result:
\begin{thm}
	\label{thm3.1.1}
	Let $s\in(1/2,1]$ and $f:\mathbb{R}^{n+1}\to \mathbb{R}$ be such that $\|\nabla_x f\|_{\infty}\lesssim 1$ and $\|\partial_t^{\frac{1}{2s}}f\|_{\ast,p_s}\lesssim 1$. Then, $f$ is $(1,\frac{1}{2s})$-Lipschitz. 
\end{thm}

\begin{rem}
\label{rem2.1}
    In the above arguments it is not clear why the restriction $s>1/2$ is necessary. In fact, relation \eqref{new_eq3.1.3} becomes trivial if $s=1/2$, since the operator $D_{n+1,s}$ is just the ordinary derivative $\partial_t$. The necessity of $s>1/2$ is required in order for estimate \cite[Equation 14]{Ho} to hold, that in our current setting reads as follows:
    \begin{equation*}
        \iint \bigg\rvert \frac{1}{|(y,u-t)|_{p_s}^{n+2s-1}}-\frac{1}{|(y,u)|_{p_s}^{n+2s-1}} \bigg\rvert \dd y \dd u \leq C |t|^{\frac{1}{2s}}. 
    \end{equation*}
    It is not difficult to prove that if $s\leq 1/2$, such finite constant $C$ does not exist and we could not proceed. In fact, the result for $s=1/2$ \textit{must} be false since, in general, it is not true that a function $f$ satisfying $\|\nabla_x f\|_{\infty}\lesssim 1$ and $\|\partial_tf\|_{\ast}\lesssim 1$ is Lipschitz.
\end{rem}

In light of the above theorem, we are able to define the so-called Lipschitz $s$-caloric capacity and the notion of removability in a more general manner as follows:
\begin{defn}
	For $E\subset \mathbb{R}^{n+1}$ compact set, its \textit{Lipschitz $s$-caloric capacity} is
	\begin{equation*}
		\Gamma_{\Theta^s}(E):=\sup  |\langle T, 1 \rangle| ,
	\end{equation*}
	where the supremum is taken among all distributions $T$ with $\text{supp}(T)\subset E$ and satisfying
	\begin{equation*}
		\|\nabla_x P_s\ast T\|_{\infty} \leq 1, \hspace{0.75cm} \|\partial^{\frac{1}{2s}}_tP_s\ast T\|_{\ast, p_s}\leq 1.
	\end{equation*}
	Such distributions will be called \textit{admissible for $\Gamma_{\Theta^s}(E)$}. If the supremum is taken only among positive Borel measures supported on $E$ satisfying the same normalization conditions, we obtain the smaller capacity $\Gamma_{\Theta^s,+}(E)$.
\end{defn}

\begin{defn}
	 A compact set $E\subset \mathbb{R}^{n+1}$ is said to be \textit{removable for Lipschitz $s$-caloric functions} if for any open subset $\Omega\subset \mathbb{R}^{n+1}$, any function $f:\mathbb{R}^{n+1}\to \mathbb{R}$ with
	\begin{equation*}
		\|\nabla_x f\|_{\infty} < \infty, \hspace{0.75cm} \|\partial_t^{\frac{1}{2s}}f\|_{\ast, p_s}<\infty,
	\end{equation*}
	satisfying the $\Theta^s$-equation in $\Omega\setminus{E}$, also satisfies such equation in $\Omega$.
\end{defn}

\section{Localization of potentials}
\label{sec3.2}

Our next goal will be to prove a so-called localization result. Previous to that, let us recall some basic notions regarding growth of measures and distributions. We say that a positive Borel measure $\mu$ in $\mathbb{R}^{n+1}$ has upper $s$-parabolic \textit{growth of degree $\rho$ \text{\normalfont{(}}with constant $C$\text{\normalfont{)}}} or simply $s$\textit{-parabolic $\rho$-growth} if there is some constant $C(n,s)>0$ such that for any $s$-parabolic ball $B(\ox,r)$,
	\begin{equation*}
		\mu\big( B(\overline{x},r) \big) \leq Cr^\rho.
	\end{equation*}
It is clear that this property is invariant if formulated using cubes instead of balls. We will be interested in a generalized version of such growth that can be defined not only for measures, but also for general distributions. To introduce such notion we present the concept of admissible function:
\begin{defn}
    Let $s\in(0,1)$. Given $\phi\in \pazocal{C}^\infty(\mathbb{R}^{n+1})$, we will say that it is an \textit{admissible} function for an $s$-parabolic cube $Q$ in $\mathbb{R}^{n+1}$ if $\text{supp}(\phi)\subset Q$ and
    \begin{equation*}
        \|\phi\|_\infty \leq 1, \qquad \|\nabla_x\phi\|_\infty \leq \ell(Q)^{-1}, \qquad \|\partial_t\phi\|_\infty \leq \ell(Q)^{-2s}, \qquad \|\Delta \phi\|_\infty \leq \ell(Q)^{-2}.
    \end{equation*}
\end{defn}

If $\phi$ is a $\pazocal{C}^{2}$ function supported on $Q\subset \mathbb{R}^{n+1}$ $s$-parabolic cube with $\|\phi\|_\infty \leq 1$, $\|\nabla_x\phi\|_\infty \leq \ell(Q)^{-1}$ and $\|\Delta \phi\|_\infty \leq \ell(Q)^{-2}$, then it is not hard to prove \cite[Remark 3.1]{HeMPr} that it also satisfies
\begin{equation*}
    \|(-\Delta)^s\phi\|_\infty \lesssim \ell(Q)^{-2s}.
\end{equation*}
\begin{defn}

We will say that a distribution $T$ in $\mathbb{R}^{n+1}$ has $s$-parabolic $n$-\textit{growth} if there exists some constant $C=C(n,s)>0$ such that, given any $s$-parabolic cube $Q\subset \mathbb{R}^{n+1}$ and any function $\phi$ admissible for $Q$, we have
\begin{equation*}
    |\langle T, \phi \rangle |\leq C \ell(Q)^n.
\end{equation*}
\end{defn}

Let observe that, as expected, admissible distributions for $\Gamma_{\Theta^s}$ satisfy a certain upper $s$-parabolic growth property analogous to that of \cite[Theorem 5.2]{HeMPr}.
\begin{thm}
	\label{thm3.2.1}
	Let $T$ be a distribution in $\mathbb{R}^{n+1}$ with
	\begin{equation*}
		\|\nabla_x P_s\ast T\|_{\infty} \leq 1, \hspace{0.75cm} \|\partial^{\frac{1}{2s}}_tP_s\ast T\|_{\ast,p_s}\leq 1.
	\end{equation*}
	Let $Q$ be a fixed $s$-parabolic cube and $\varphi$ an admissible function for $Q$. Then, if $R$ is any $s$-parabolic cube with $\ell(R)\leq \ell(Q)$ and $\phi$ is admissible for $R$, we have $|\langle \varphi T, \phi \rangle|\lesssim \ell(R)^{n+1}$.
\end{thm}
\begin{proof}
	Is a direct consequence of \cite[Theorem 5.2]{HeMPr}.
\end{proof}

Let us turn our attention to the localization result we are interested in. This type of result will ensure that if a distribution $T$ satisfies
\begin{equation*}
	\|\nabla_x P_s \ast T\|_{\infty}\leq 1, \qquad \|\partial_t^{\frac{1}{2s}}P_s\ast T\|_{\ast,p_s}\leq 1,
\end{equation*}
then
\begin{equation*}
	\|\nabla_x P_s \ast \varphi T\|_{\infty}\lesssim 1, \qquad \|\partial_t^{\frac{1}{2s}}P_s\ast \varphi T\|_{\ast,p_s}\lesssim 1,
\end{equation*}
where $\varphi$ is an admissible function for some $s$-parabolic cube. That is, $\varphi T$, the \textit{localized} distribution, is admissible for $\Gamma_{\Theta^s}(Q)$. We begin by estimating $\|\nabla_x P_s \ast \varphi T\|_{\infty}$ in the following result (stated in a more general way taking into account Theorem \ref{thm3.1.1}).

\begin{lem}
	\label{lem3.2.2}
	Let $T$ be a distribution in $\mathbb{R}^{n+1}$ satisfying
	\begin{equation*}
		\|\nabla_x P_s \ast T\|_{\infty}\leq 1, \qquad \|P_s\ast T\|_{\text{\normalfont{Lip}}_{\frac{1}{2s},t}}\leq 1.
	\end{equation*}
	Let $Q$ be an $s$-parabolic cube and $\varphi$ admissible function for $Q$. Then,
	\begin{equation*}
		\|\nabla_x P_s \ast \varphi T\|_{\infty}\lesssim 1.
	\end{equation*}
\end{lem}
\begin{proof}
	Let us begin by recalling the following product rule that can be deduced directly from the definition of $(-\Delta)^s$ (see \cite[Eq.(4.1)]{RSe}, for example)
	\begin{equation*}
		(-\Delta)^s(fg) = f(-\Delta)^sg+g(-\Delta)^sf-I_s(f,g),
	\end{equation*}
	where $I_s$ is defined as
	\begin{equation*}
		I_s(f,g)\simeq \int_{\mathbb{R}^n}\frac{(f(x)-f(y))(g(x)-g(y))}{|x-y|^{n+2s}}\dd y.
	\end{equation*}
	Then, for some constant $c$ to be fixed later on we have
	\begin{align*}
		\Theta^{s}\big( \varphi(P_s\ast T-c) \big) &= \varphi\Theta^s(P_s\ast T-c)+(P_s\ast T-c)\Theta^s\varphi-I_s(\varphi, P_s\ast T-c)\\
		&=\varphi T+(P_s\ast T-c)\Theta^s\varphi-I_s(\varphi,P_s\ast T),
	\end{align*}
	and therefore
	\begin{align*}
		\nabla_x P_s\ast \varphi T &= \nabla_x\big( \varphi(P_s\ast T -c) \big)-\nabla_x P_s \ast \big( (P_s\ast T-c)\Theta^s\varphi \big)+\nabla_xP_s\ast I_s(\varphi, P_s\ast T)\\
		&=:\text{I}+\text{II}+\text{III}.
	\end{align*}
	To estimate the $L^\infty$ norm of the previous terms write $Q:=Q_1\times I_Q$, where $Q_1\subset \mathbb{R}^n$ is an euclidean $n$-dimensional cube of side $\ell(Q)$, and $I_Q\subset \mathbb{R}$ is an interval of length $\ell(Q)^{2s}$. Choose $c:=P_s\ast T(\ox_Q)$, with $\ox_Q:=(x_Q,t_Q)$ the center of $Q$. This way, since $P_s\ast T$ satisfies a $(1,\frac{1}{2s})$-Lipschitz property, for any $\ox=(x,t)\in Q$ we have 
	\begin{align}
		|P_s\ast T (\ox)&-P_s\ast T (\ox_Q)|\nonumber \\
		&\leq |P_s\ast T (x,t)-P_s\ast T (x_Q,t)|+|P_s\ast T (x_Q,t)-P_s\ast T (x_Q,t_Q)|\lesssim \ell(Q). \label{eq3.2.1}
	\end{align}
	Using this estimate, we are able to bound term $\text{I}$ as follows:
	\begin{align*}
		\|\nabla_x\big( \varphi(P_s\ast T &-c) \big)\|_\infty \\
		& \leq \|\nabla_x\varphi\|_{\infty}\|P_s\ast T-P_s\ast T(\ox_Q)\|_{\infty,Q}+\|\varphi\|_\infty\|\nabla_xP_s\ast T\|_{\infty}\lesssim 1.
	\end{align*}
	Let us move on by studying $\text{II}$. We name $g_1:=(P_s\ast T - P_s\ast T(\ox_Q))\Theta^s\varphi$ and observe that by relation \eqref{eq3.2.1} and the admissibility of $\varphi$ (see \cite[Remark 3.1]{HeMPr}), we have $\|g_1\|_\infty\lesssim \ell(Q)^{-2s+1}$. Let us now proceed by computing,
	\begin{align*}
		|\text{II}(\ox)|&\leq \int_{2Q}|\nabla_x P_s(\ox-\oy)||g_1(\oy)|\dd\oy+\int_{\mathbb{R}^{n+1}\setminus{2Q}}|\nabla_x P_s(\ox-\oy)||g_1(\oy)|\dd\oy\\
		&=:\text{II}_1(\ox)+\text{II}_2(\ox).
	\end{align*}
	Let $R$ be the $s$-parabolic cube centered at $\ox$ with length $\ell(Q)$. We begin by studying $\text{II}_1$. Assume $\ox \in 4Q$ so that by the first estimate of \cite[Theorem 2.2]{HeMPr} we get
	\begin{align*}
		\text{II}_1(\ox)&\lesssim \|g_1\|_{\infty}\int_{2Q}\frac{|x-y||t-u|}{|\ox-\oy|_{p_s}^{n+2s+2}}\dd \oy \lesssim \ell(Q)^{-2s+1}\int_{8R}\frac{\dd y}{|\ox-\oy|_{p_s}^{n+1}}\\
		&\lesssim \ell(Q)^{-2s+1}\bigg( \int_{8R_1}\frac{\dd y}{|x-y|^{n-\varepsilon}} \bigg)\bigg( \int_{8^{2s}I_R}\frac{\dd u}{|t-u|^{\frac{\varepsilon+1}{2s}}} \bigg)\lesssim 1,
	\end{align*}
	where we have chosen $0<\varepsilon<2s-1$. On the other hand, if $\ox\notin 4Q$,
	\begin{equation*}
		\text{II}_1(\ox)\lesssim \ell(Q)^{-2s+1}\int_{2Q}\frac{\dd y}{|\ox-\oy|_{p_s}^{n+1}}\leq \ell(Q)^{-2s+1}\frac{|2Q|}{\ell(Q)^{n+1}}\lesssim 1.
	\end{equation*}
	We move on by studying $\text{II}_2$. Let us first consider the case in which $\ox\in \mathbb{R}^n\times 2I_Q$. Since $\text{supp}(g_1)\subset \mathbb{R}^n\times I_Q$,
	\begin{align*}
		|\text{II}_2(\ox)|&\lesssim \ell(Q)^{-2s+1}\int_{(\mathbb{R}^n\times I_Q)\setminus{2Q}}\frac{|t-u|}{|\ox-\oy|_{p_s}^{n+2s+1}}\dd\oy\\
		&=\ell(Q)^{-2s+1}\int_{[(\mathbb{R}^n\times I_Q)\setminus{2Q}]\setminus{8R}}\frac{|t-u|}{|\ox-\oy|_{p_s}^{n+2s+1}}\dd\oy\\
		&\hspace{5cm}+\ell(Q)^{-2s+1}\int_{[(\mathbb{R}^n\times I_Q)\setminus{2Q}]\cap 8R}\frac{|t-u|}{|\ox-\oy|_{p_s}^{n+2s+1}}\dd\oy\\
		&\lesssim \ell(Q)\int_{\mathbb{R}^{n+1}\setminus{8R}}\frac{\dd\oy}{|\ox-\oy|_{p_s}^{n+2s+1}}+\ell(Q)^{-2s+1}\int_{8R}\frac{\dd\oy}{|\ox-\oy|_{p_s}^{n+1}}\lesssim 1.
	\end{align*}
	Now assume that $\ox\notin \mathbb{R}^{n}\times 2I_Q$. in this case, set $\ell_t:=\text{dist}_{p_s}(\ox,\mathbb{R}^n\times I_q)\gtrsim \ell(Q)$. Then, if $A_j:=Q(\ox,2^j\ell_t)\setminus{Q(\ox,2^{j-1}\ell_t)}$ for $j\geq 0$,
	\begin{align*}
		|\text{II}_2(\ox)|&\lesssim \ell(Q)^{-2s+1}\sum_{j=0}^\infty \int_{A_j\cap [(\mathbb{R}^n\times I_Q)\setminus{2Q}]}\frac{\dd \oy}{|\ox-\oy|_{p_s}^{n+1}}\\
		&\lesssim \ell(Q)^{-2s+1}\sum_{j=0}^\infty \frac{\ell(Q)^{2s}\big[2^j\ell(Q) \big]^n}{\big[ 2^j\ell(Q)\big]^{n+1}}\lesssim 1,
	\end{align*}
	and we are done with $\text{II}$. Finally, we study $\text{III}$, and we notice that if we prove that the function
	\begin{equation*}
		g_2(\ox):=I_s(\varphi,P_s\ast T)(\ox)=c_{n,s}\int_{\mathbb{R}^n}\frac{(\varphi(x,t)-\varphi(y,t))(P_s\ast T(x,t)-P\ast T(y,t))}{|x-y|^{n+2s}}\dd y
	\end{equation*}
	is such that $\|g_2\|_{\infty}\lesssim \ell(Q)^{-2s+1}$ we will be done, since we would be able to repeat the same arguments done for $\text{II}$. To do so, we distinguish two cases: if $\ox\notin 2Q$, then
	\begin{align*}
		|g_2(\ox)|&\lesssim \int_{Q}\frac{|\varphi(y,t)||P_s\ast T(x,t)-P_s\ast T(y,t)|}{|x-y|^{n+2s}}\dd y\\
		&\lesssim \|\varphi\|_{\infty}\int_Q\frac{\|\nabla_xP_s\ast T(\cdot,t)\|_{\infty}}{|x-y|^{n+2s-1}}\dd y \lesssim \frac{\ell(Q)^n}{\ell(Q)^{n+2s-1}}=\ell(Q)^{-2s+1}.
	\end{align*}
	If $\ox\in 2Q$, then $|g_2(\ox)|$ is bounded by
	\begin{align*}
		\int_{4Q}&\frac{|\varphi(x,t)-\varphi(y,t)||P_s\ast T(x,t)-P\ast T(y,t)|}{|x-y|^{n+2s}}\dd y\\
		&\hspace{3.5cm}+\int_{\mathbb{R}^n\setminus{4Q}}\frac{|\varphi(x,t)-\varphi(y,t)||P_s\ast T(x,t)-P\ast T(y,t)|}{|x-y|^{n+2s}}\dd y\\
		&\leq  \int_{Q(x,8\ell(Q))}\frac{\|\nabla_x\varphi\|_\infty \dd y}{|x-y|^{n+2s-2}}+\int_{\mathbb{R}^{n}\setminus{Q(x,\ell(Q)/2)}}\frac{2\|\varphi\|_\infty \dd y}{|x-y|^{n+2s-1}} \lesssim \ell(Q)^{-2s+1},
	\end{align*}
	and we are done.
\end{proof}
The next goal is to obtain an analogous result for the potential $\partial_t^{\frac{1}{2s}}P_s\ast \varphi T$. Our arguments are inspired by those in \cite[\textsection 3]{MPrT}. To this end, we first prove an auxiliary result that generalizes \cite[Lemma 3.5]{MPrT}.
\begin{lem}
	\label{lem3.2.3}
	Let $Q=Q_1\times I_Q\subset \mathbb{R}^{n+1}$ be an $s$-parabolic cube and $g$ a function supported on $\mathbb{R}^n\times I_Q$ such that $\|g\|_{\infty}\lesssim \ell(Q)^{-2s+1}$. Then,
	\begin{equation*}
		\|P_s\ast g\|_{\text{\normalfont{Lip}}_{\frac{1}{2s},t}}\lesssim 1.
	\end{equation*}
\end{lem}
\begin{proof}
	Fix $\ox=(x,t), \widetilde{x}=(x,r)$ as well as a function $g$ with $\|g\|\lesssim \ell(Q)^{-2s+1}$, supported on $\mathbb{R}^{n}\times I_Q$. If $\ell^{2s}:=|t-r|$ and $R$ is the $s$-parabolic cube centered at $\ox$ with side length $\max\{\ell(Q),\ell\}$,
	\begin{align*}
		|P_s&\ast g (\ox)-P_s\ast g(\widetilde{x})|\\
		&\lesssim \frac{1}{\ell(Q)^{2s-1}}\int_{(\mathbb{R}^n\times I_Q)\cap 4R}|P_s(x-z,t-u)-P_s(x-z,r-u)|\dd z \dd u\\
		&\hspace{0.75cm}+\frac{1}{\ell(Q)^{2s-1}}\int_{(\mathbb{R}^n\times I_Q)\setminus 4R}|P_s(x-z,t-u)-P_s(x-z,r-u)|\dd z \dd u =:\text{I}+\text{II}.
	\end{align*}
	Let us first deal with $\text{II}$:
	\begin{align*}
		\text{II}\lesssim \frac{\ell^{2s}}{\ell(Q)^{2s-1}}&\int_{(\mathbb{R}^n\times I_Q)\setminus 4R}\frac{\dd z \dd u}{|(x-z,t-u)|_{p_s}^{n+2s}}\\
		&\leq \frac{\ell^{2s}}{\ell(Q)^{2s-1}}\int_{\mathbb{R}^n\setminus{4R_1}}\frac{\dd z}{|x-z|^{n+1}}\int_{I_Q\setminus{4^{2s}I_R}}\frac{\dd u}{|t-u|^{1-\frac{1}{2s}}}.
	\end{align*}
	For the spatial integral we simply integrate using polar coordinates for example, and for the temporal integral we use that in $I_Q\setminus{4^{2s}I_R}$, $|t-u|\geq \max\{\ell(Q),\ell\}^{2s}$, and then it can be bounded by
	\begin{align*}
		\frac{\ell^{2s}}{\ell(Q)^{2s-1}}\cdot\frac{1}{\max\{\ell(Q),\ell\}}\cdot\frac{\ell(Q)^{2s}}{\max\{\ell(Q),\ell\}^{2s-1}} \leq \frac{\ell^{2s}\ell(Q)}{\ell(Q)\ell^{2s-1}}=|t-r|^{\frac{1}{2s}}.
	\end{align*}
	Regarding $\text{I}$, let $S$ be the $s$-parabolic cube centered at $\widetilde{x}$ with side length $\max\{\ell(Q),\ell\}$, so that $4R\subset 8S$. If we assume $\ell\geq \ell(Q)$,
	\begin{align*}
		\text{I}\leq \frac{1}{\ell(Q)^{2s-1}}\bigg[ \int_{(\mathbb{R}^n\times I_Q)\cap 4R}\frac{\dd z \dd u}{|(x-z,t-u)|_{p_s}^{n}} + \int_{(\mathbb{R}^n\times I_Q)\cap 8S}\frac{\dd z \dd u}{|(x-z,r-u)|_{p_s}^{n}} \bigg].
	\end{align*}
	We study the first integral, the second can be estimated analogously. We compute,
	\begin{align*}
		\int_{(\mathbb{R}^n\times I_Q)\cap 4R}\frac{\dd z \dd u}{|(x-z,t-u)|_{p_s}^{n+2s}}&\leq \int_{4R_1}\frac{\dd z}{|x-z|^{n-1}}\int_{I_Q\cap4^{2s}I_R}\frac{\dd u}{|t-u|^{\frac{1}{2s}}}\\
		&\lesssim \max\{\ell(Q),\ell\}[\ell(Q)^{2s}]^{1-\frac{1}{2s}}=\ell \cdot \ell(Q)^{2s-1},
	\end{align*}
	and with this we conclude $\text{I}\lesssim \ell$. If $\ell\leq \ell(Q)$, denote $R'$ and $S'$ be $s$-parabolic cubes of side length $\ell$ centered at $\ox$ and $\widetilde{x}$ respectively, so that
	\begin{align*}
		\text{I}\leq \frac{1}{\ell(Q)^{2s-1}}&\bigg[ \int_{(\mathbb{R}^n\times I_Q)\cap 2R'}\frac{\dd z \dd u}{|(x-z,t-u)|_{p_s}^{n}} + \int_{(\mathbb{R}^n\times I_Q)\cap 4S'}\frac{\dd z \dd u}{|(x-z,r-u)|_{p_s}^{n}} \bigg]\\
		&+\frac{1}{\ell(Q)^{2s-1}}\int_{(\mathbb{R}^n\times I_Q)\cap (4R\setminus 2R')}|P_s(x-z,t-u)-P_s(x-z,r-u)|\dd z \dd u.
	\end{align*}
	The first two summands satisfy being controlled by above by $\ell$ by an analogous argument to that given for the case $\ell\geq \ell(Q)$. The last term it can be estimated as follows:
	\begin{align*}
		\frac{\ell^{2s}}{\ell(Q)^{2s-1}}&\int_{(\mathbb{R}^n\times I_Q)\cap (4R\setminus 2R')}\frac{\dd z \dd u}{|(x-z,t-u)|_{p_s}^{n+2s}}\\
		&\leq\frac{\ell^{2s}}{\ell(Q)^{2s-1}}\int_{\mathbb{R}^n\setminus{2R_1'}}\frac{\dd z}{|x-z|^{n+2s-1}}\int_{I_Q\cap(4^{2s}I_R\setminus{2^2s}I_{R'})}\frac{\dd u}{|t-u|^{\frac{1}{2s}}}\\
		&\lesssim \frac{\ell}{\ell(Q)^{2s-1}}\int_{4^{2s}I_R}\frac{\dd u}{|t-u|^{\frac{1}{2s}}}\lesssim \frac{\ell}{\ell(Q)^{2s-1}}  [\ell(Q)^{2s}]^{1-\frac{1}{2s}} = |t-r|^{\frac{1}{2s}},
	\end{align*}
	and we are done.
\end{proof}
The previous lemma allows us to prove a weaker localization-type result, where we ask for potentials to satisfy, explicitly, a $(1,\frac{1}{2s})$-Lipschitz property, instead of asking for a $\text{BMO}_{p_s}$ estimate over its $\partial_{t}^{\frac{1}{2s}}$ derivative.
\begin{lem}
	\label{lem3.2.4}
	Let $Q\subset \mathbb{R}^{n+1}$ be an $s$-parabolic cube and $\varphi$ admissible for $Q$. Let $T$ be a distribution with $\|\nabla_x P_s\ast T\|_\infty \leq 1$ and $\|P_s\ast T\|_{\text{\normalfont{Lip}}_{\frac{1}{2s},t}}\leq 1$. Then
	\begin{equation*}
		\|P_s\ast \varphi T\|_{\text{\normalfont{Lip}}_{\frac{1}{2s},t}}\lesssim 1.
	\end{equation*}
\end{lem}
\begin{proof}
	We already know from the proof of Lemma \ref{lem3.2.2} that the following identity holds:
	\begin{equation*}
		\Theta^{s}\big( \varphi(P_s\ast T-c) \big) =\varphi T+(P_s\ast T-c)\Theta^s\varphi-I_s(\varphi,P_s\ast T),
	\end{equation*}
	and then
	\begin{align*}
		P_s\ast \varphi T = \varphi (P_s\ast T - c)-P_s\ast\big( (P_s\ast T-c)\Theta^s\varphi \big) + P_s\ast I_s(\varphi, P_s\ast T).
	\end{align*}
	Set $\ox=(x,t), \widetilde{x}=(x,r)$. Then,
	\begin{align}
		P_s\ast \varphi T(\ox) &- P_s\ast \varphi T(\widetilde{{x}})\nonumber \\
		&=\varphi(\ox)(P_s\ast T(\ox)-c)-\varphi(\widetilde{x})(P_s\ast T(\widetilde{x})-c)\label{eq3.2.2}\\
		&\hspace{1cm}+P_s\ast I_s(\varphi, P_s\ast T)(\ox)-P_s\ast I_s(\varphi, P_s\ast T)(\widetilde{x})\label{eq3.2.3}\\
		&\hspace{1cm}-P_s\ast\big( (P_s\ast T-c)\Theta^s\varphi \big)(\ox)+P_s\ast\big( (P_s\ast T-c)\Theta^s\varphi \big)(\widetilde{x}).\label{eq3.2.4}
	\end{align}
	We study \eqref{eq3.2.2} as follows: if $\ox,\widetilde{x}\notin Q$ we have that the previous difference is null. If $\ox\in Q$, choosing $c:=P_s\ast T(\ox_Q)$ with $\ox_Q$ the center of $Q$, we define $\widetilde{x}'$ as follows:
	\begin{enumerate}
		\item[\textit{i}.] if $\widetilde{x}\in Q$, $\widetilde{x}':=\widetilde{x}$,
		\item[\textit{ii}.] if $\widetilde{x}\notin Q$, consider $\widetilde{x}':=(x,r')\in 2Q\setminus{Q}$ with the property $|\widetilde{x}'-\ox|_{p_s}\leq |\widetilde{x}-\ox|_{p_s}$. 
	\end{enumerate}
	Either way we get $\varphi(\widetilde{x})(P_s\ast T(\widetilde{x})-c)= \varphi(\widetilde{x}')(P_s\ast T(\widetilde{x}')-c)$. So by \eqref{eq3.2.1} and the $\frac{1}{2s}$-Lipschitz property with respect to $t$ of $P_s\ast T$ to obtain,
	\begin{align*}
		\big\rvert \varphi(\ox)(P_s\ast T(\ox)-c)&-\varphi(\widetilde{x})(P_s\ast T(\widetilde{x})-c) \big\rvert \\
		&\leq |\varphi(\widetilde{x}')-\varphi(\ox)||P_s\ast T (\widetilde{x}')-c|+|\varphi(\ox)||P_s\ast T (\widetilde{x}')-P_s\ast T(\ox)|\\
		&\lesssim \frac{|r'-t|}{\ell(Q)^{2s}}\ell(Q)+|r'-t|^{\frac{1}{2s}}=\bigg[ \frac{|r'-t|^{\frac{2s-1}{2s}}}{\ell(Q)^{2s-1}}+1 \bigg]|r'-t|^{\frac{1}{2s}}\\
		&\lesssim |r'-t|^{\frac{1}{2s}}\leq |r-t|^{\frac{1}{2s}}.
	\end{align*}
	In order to estimate the remaining differences \eqref{eq3.2.3} and $\eqref{eq3.2.4}$ we name $g_1:=(P_s\ast T-c)\Theta^s\varphi$ and $g_2:=I_s(\varphi,P_s\ast T)$. Such functions have already appeared in the proof of Lemma \ref{lem3.2.2} and satisfy $\text{supp}(g_j)\subset \mathbb{R}^{n}\times I_Q$ and $\|g_j\|_\infty \lesssim \ell(Q)^{-2s+1}$, for $j=1,2$. By a direct application of Lemma \ref{lem3.2.3} we deduce that both differences are bounded by $|r-t|^{\frac{1}{2s}}$ up to a constant, and we are done.
\end{proof}
We move on by proving two additional auxiliary lemmas which will finally allow us to deduced the desired localization theorem.

\begin{lem}
	\label{lem3.2.5}
	Let $T$ be a distribution in $\mathbb{R}^{n+1}$ such that $\|\nabla_xP_s\ast T\|_\infty \leq 1$ and $\|\partial_t^{\frac{1}{2s}}P_s\ast T\|_{\ast,p_s}\leq 1$. Let $Q,R\subset \mathbb{R}^{n+1}$ be $s$-parabolic cubes such that $Q\subset R$. Then, if $\varphi$ is admissible for $Q$,
	\begin{equation*}
		\int_R \big\rvert \partial_t^{\frac{1}{2s}}P_s\ast \varphi T (\ox) \big\rvert \dd \ox \lesssim \ell(R)^{n+2s}.
	\end{equation*}
\end{lem}
\begin{proof}
	We know that
	\begin{align*}
		P_s\ast \varphi T = \varphi (P_s\ast T - c)-P_s\ast\big( (P_s\ast T-c)\Theta^s\varphi \big) + P_s\ast I_s(\varphi, P_s\ast T),
	\end{align*}
	and we choose $c:=P_s\ast T(\ox_Q)$, with $\ox_Q$ the center of $Q$. Name $g_1:=(P_s\ast T-c)\Theta^s\varphi$ and $g_2:=I_s(\varphi,P_s\ast T)$ so that
	\begin{equation}
		\label{eq3.2.5}
		\partial_t^{\frac{1}{2s}}P_s\ast \varphi T = \partial_t^{\frac{1}{2s}}\big[ \varphi(P_s\ast T - c) \big] - \partial_t^{\frac{1}{2s}}P_s\ast g_1 + \partial_t^{\frac{1}{2s}}P_s\ast g_2.
	\end{equation}
	Begin by noticing that $\forall \ox\in\mathbb{R}^{n+1}$ and $i=1,2$, if $Q_{\ox}=Q_{1,\ox}\times I_{Q,\ox}$ is the $s$-parabolic cube of side length $\ell(Q)$ and center $\ox$,
	\begin{align*}
		\big\rvert \partial_t^{\frac{1}{2s}}&P_s\ast g_i(\ox)\big\rvert \lesssim \ell(Q)^{-2s+1}\int_{\mathbb{R}^n\times I_Q}\big\rvert \partial_t^{\frac{1}{2s}}P_s(\ox-\oy)\big\rvert\dd \oy\\
		&= \ell(Q)^{-2s+1}\bigg[ \int_{(\mathbb{R}^n\times I_Q)\cap 4Q_{\ox}}\big\rvert \partial_t^{\frac{1}{2s}}P_s(\ox-\oy)\big\rvert\dd \oy + \int_{(\mathbb{R}^n\times I_Q)\setminus 4Q_{\ox}}\big\rvert \partial_t^{\frac{1}{2s}}P_s(\ox-\oy)\big\rvert\dd \oy \bigg].
	\end{align*}
	Regarding the first integral, assuming $n>1$ and applying \cite[Theorem 2.4]{HeMPr} with $\beta:=\frac{1}{2s}$, we have
	\begin{align*}
		\int_{(\mathbb{R}^n\times I_Q)\cap 4Q_{\ox}}\big\rvert \partial_t^{\frac{1}{2s}}P_s(\ox-\oy)\big\rvert\dd \oy &\lesssim \int_{(\mathbb{R}^n\times I_Q)\cap 4Q_{\ox}}\frac{\dd \oy}{|x-y|^{n-2s}|\ox-\oy|_{p_s}^{2s+1}}\\
		& \leq \int_{4Q_{1,\ox}}\frac{\dd y}{|x-y|^{n-2s+\varepsilon}}\int_{4^{2s}I_{Q,\ox}}\frac{\dd r}{|t-r|^{\frac{2s+1-\varepsilon}{2s}}}\\
		&\lesssim \ell(Q)^{2s-\varepsilon}\ell(Q)^{-1+\varepsilon}=\ell(Q)^{2s-1},
	\end{align*}
	where we have chosen $1<\varepsilon<2s$. If $n=1$, by \cite[Theorem 2.4]{HeMPr} we also know that for any $\alpha\in(2s-1,4s)$,
	\begin{align*}
		\int_{(\mathbb{R}^n\times I_Q)\cap 4Q_{\ox}}\big\rvert \partial_t^{\frac{1}{2s}}P_s(\ox-\oy)\big\rvert\dd \oy &\lesssim \int_{(\mathbb{R}^n\times I_Q)\cap 4Q_{\ox}}\frac{\dd \oy}{|x-y|^{1-2s+\alpha}|\ox-\oy|_{p_s}^{2s+1-\alpha}}.
	\end{align*}
	So choosing $1-\alpha<\varepsilon<2s-\alpha$ we can carry out a similar argument and deduce the desired bound. For the second integral we also distinguish whether if $n>1$ or $n=1$ (we will only give the details for the case $n>1$). Defining $A_j:=\{\oy\,:\, 2^j\ell(Q)\leq \text{dist}_{p_s}(\ox,\oy)\leq 2^{j+1}\ell(Q)\}$ for $j\geq 1$, we have
	\begin{align*}
		\int_{(\mathbb{R}^n\times I_Q)\setminus 4Q_{\ox}}&\big\rvert \partial_t^{\frac{1}{2s}}P_s(\ox-\oy)\big\rvert\dd \oy \lesssim \sum_{j=1}^{\infty} \int_{(\mathbb{R}^n\times I_Q)\cap A_j} \frac{\dd \oy}{|x-y|^{n-2s}|\ox-\oy|_{p_s}^{2s+1}}\\
		&\hspace{-0.75cm}\leq \ell(Q)^{2s} \sum_{j=1}^{\infty} \int_{(\mathbb{R}^n\times I_Q)\cap A_j} \frac{\dd y}{|x-y|^{n+1}}\lesssim \ell(Q)^{2s} \sum_{j=1}^{\infty} \frac{(2^j\ell(Q))^n}{(2^j\ell(Q))^{n+1}} = \ell(Q)^{2s-1}.
	\end{align*}
	Therefore we obtain
	\begin{equation*}
		\big\rvert \partial_t^{\frac{1}{2s}}P_s\ast g_i(\ox)\big\rvert\lesssim 1, \qquad \forall \ox \in \mathbb{R}^{n+1} \,\text{ and }\, i=1,2.
	\end{equation*}
	Now, returning to \eqref{eq3.2.5} and integrating both sides over $R$, we get
	\begin{equation*}
		\int_R\big\rvert\partial_t^{\frac{1}{2s}}P_s\ast \varphi T (\ox)\big\rvert \dd \ox = \int_R\big\rvert\partial_t^{\frac{1}{2s}}\big[ \varphi(P_s\ast T - c) \big](\ox)\big\rvert \dd \ox +\ell(R)^{n+2s}.
	\end{equation*}
	So we are left to study the integral
	\begin{equation*}
		\int_R\big\rvert\partial_t^{\frac{1}{2s}}\big[ \varphi\big(P_s\ast T - P_s\ast T(\ox_Q)\big) \big](\ox)\big\rvert \dd \ox.
	\end{equation*}
	To this end, take $\psi_Q$ test function with $\chi_Q\leq \psi_Q \leq \chi_{2Q}$, with $\|\nabla_x \psi_Q\|_{\infty}\leq \ell(Q)^{-1}$, $\|\partial_t \psi_Q\|_\infty\leq \ell(Q)^{-2s}$ and $\|\Delta \psi_Q\|_\infty\leq \ell(Q)^{-2}$. We also write for locally integrable function $F$,
	\begin{equation*}
		m_{\psi_Q}(F):=\frac{\int F \psi_Q}{\int \psi_Q}.
	\end{equation*}
	Using the product rule we also know that
	\begin{align*}
		\partial_t^{\frac{1}{2s}}(fg)(t)=g\partial_t^{\frac{1}{2s}}f(t)+f\partial_t^{\frac{1}{2s}}g(t)+c_{s}\int_{\mathbb{R}}\frac{\big( f(r)-f(t) \big)\big( g(r)-g(t) \big)}{|r-t|^{1+\frac{1}{2s}}}\dd r.
	\end{align*}
	So choosing $f:=\varphi(x,\cdot)$ and $g:=P_s\ast T(x,\cdot)-P_s\ast T(\ox_Q)$,
	\begin{align*}
		\big\rvert\partial_t^{\frac{1}{2s}}\big[ \varphi\big(P_s&\ast T - P_s\ast T(\ox_Q)\big) \big](x,t)\big\rvert\\
		&=\big( P_s\ast T (x,t) - P_s\ast T (\ox_Q) \big)\partial_t^{\frac{1}{2s}}\varphi (x,t)+\varphi(x,t)\partial_t^{\frac{1}{2s}}P_s\ast T(x,t)\\
		&\hspace{2cm}+\int_{\mathbb{R}}\frac{\big( \varphi(x,r)-\varphi(x,t) \big)\big( P_s\ast T(x,r)-P_s\ast T(x,t) \big)}{|r-t|^{1+\frac{1}{2s}}}\dd r\\
		&=\big( P_s\ast T (x,t) - P_s\ast T (\ox_Q) \big)\partial_t^{\frac{1}{2s}}\varphi (x,t)\\
		&\hspace{2cm}+\varphi(x,t)\big( \partial_t^{\frac{1}{2s}}P_s\ast T(x,t)-m_{\psi_Q}\big( \partial_t^{\frac{1}{2s}}P_s\ast T \big) \big)\\
		&\hspace{2cm}+\int_{\mathbb{R}}\frac{\big( \varphi(x,r)-\varphi(x,t) \big)\big( P_s\ast T(x,r)-P_s\ast T(x,t) \big)}{|r-t|^{1+\frac{1}{2s}}}\dd r\\
		&\hspace{2cm} +\varphi(x,t)m_{\psi_Q}\big( \partial_t^{\frac{1}{2s}}P_s\ast T \big)=: \text{I}(\ox)+\text{II}(\ox)+\text{III}(\ox)+\text{IV}(\ox).
	\end{align*}
	To study $\text{I}$, observe that $\text{supp}(\partial_t^{\frac{1}{2s}}\varphi)\subset Q_1\times \mathbb{R}$. If $(x,t)\in Q_1\times 2^{2s}I_Q$ we have
	\begin{align*}
		|\partial_t^{\frac{1}{2s}}\varphi(x,t)|&\leq \int_{\mathbb{R}}\frac{|\varphi(x,r)-\varphi(x,t)|}{|r-t|^{1+\frac{1}{2s}}}\dd r\\
		&\lesssim \frac{1}{\ell(Q)^{2s}}\int_{4^{2s}I_Q}\frac{|r-t|}{|r-t|^{1+\frac{1}{2s}}}\dd r+\int_{\mathbb{R}\setminus{4^{2s}I_Q}}\frac{\dd r}{|r-t|^{1+\frac{1}{2s}}}\lesssim \frac{1}{\ell(Q)}.
	\end{align*}
	If on the other hand $(x,t)\notin Q_1\times 2^{2s}I_Q$,
	\begin{align*}
		|\partial_t^{\frac{1}{2s}}\varphi(x,t)|\leq \int_{I_Q}\frac{|\varphi(x,r)|}{|r-t|^{1+\frac{1}{2s}}}\dd r \lesssim \frac{\ell(Q)^{2s}}{|t-t_Q|^{1+\frac{1}{2s}}}.
	\end{align*}
	So in any case we are able to deduce
	\begin{align*}
		|\partial_t^{\frac{1}{2s}}\varphi(x,t)|\lesssim \frac{\ell(Q)^{2s}}{\ell(Q)^{2s+1}+|t-t_Q|^{1+\frac{1}{2s}}}, \qquad \forall \ox\in Q_1\times \mathbb{R}.
	\end{align*}
	Then, we infer that for any $x\in Q_1$, by the $(1,\frac{1}{2s})$-Lipschitz property of $P_s\ast T$,
	\begin{equation*}
		|\text{I}(\ox)|\lesssim \frac{\ell(Q)^{2s}\big( \ell(Q)+|t-t_Q|^{\frac{1}{2s}} \big)}{\ell(Q)^{2s+1}+|t-t_Q|^{1+\frac{1}{2s}}},
	\end{equation*}
	and therefore
	\begin{align*}
		\int_{R}|\text{I}(\ox)|\dd\ox &\lesssim \int_{Q_1}\int_{|t-t_Q|^{\frac{1}{2s}}\leq 2\ell(R)}\frac{\ell(Q)^{2s}\big( \ell(Q)+|t-t_Q|^{\frac{1}{2s}}\big)}{\ell(Q)^{2s+1}+|t-t_Q|^{1+\frac{1}{2s}}}\dd t \dd x\\
		&=\ell(Q)^{n+2s}\int_{|u|\leq \big(2\frac{\ell(R)}{\ell(Q)}\big)^{2s}}\frac{1+|u|^{\frac{1}{2s}}}{1+|u|^{1+\frac{1}{2s}}}\dd u\\
		&\lesssim \ell(Q)^{n+2s}\bigg( 1+ \int_{1\leq |u| \leq \big(2\frac{\ell(R)}{\ell(Q)}\big)^{2s} } \frac{1+|u|^{\frac{1}{2s}}}{1+|u|^{1+\frac{1}{2s}}}\dd u \bigg)\\
		&\lesssim \ell(Q)^{n+2s}\bigg( 1+\bigg(\frac{\ell(R)}{\ell(Q)}\bigg)^{2s}+\log\frac{\ell(R)}{\ell(Q)} \bigg)\\
		&\lesssim \ell(R)^{n+2s}\bigg( 1+\bigg(\frac{\ell(Q)}{\ell(R)}\bigg)^{2s}\log\frac{\ell(R)}{\ell(Q)} \bigg)\lesssim \ell(R)^{n+2s},
	\end{align*}
	where in the last step we have used that the function $x\mapsto x\log(1/x)$ is bounded for $0<x\leq 1$. We move on to $\text{II}$. As discussed in \cite{MPrT}, one has
	\begin{equation}
		\label{eq3.2.6}
		\big\rvert m_{\psi_Q}\big( \partial_t^{\frac{1}{2s}}P_s\ast T\big) - \big( \partial_t^{\frac{1}{2s}}P_s\ast T\big)_Q \big\rvert \lesssim \|\partial_t^{\frac{1}{2s}}P_s\ast T\|_{\ast,p_s}\leq 1.
	\end{equation}
	One way to see this is to observe that
	\begin{align*}
		\big\rvert m_{\psi_Q}\big( \partial_t^{\frac{1}{2s}}P_s\ast T\big) &- \big( \partial_t^{\frac{1}{2s}}P_s\ast T\big)_Q \big\rvert \\
		&= \frac{1}{|Q|}\bigg\rvert \int_Q\big( \partial_t^{\frac{1}{2s}}P_s\ast T (\ox)- m_{\psi_Q}\big( \partial_t^{\frac{1}{2s}}P_s\ast T\big)\big)\dd\ox \bigg\rvert\\
		&\leq \frac{\|\psi_Q\|_1}{|Q|}\cdot\frac{1}{\|\psi_Q\|_1}\int_{2Q}\big\rvert \partial_t^{\frac{1}{2s}}P_s\ast T (\ox)- m_{\psi_Q}\big( \partial_t^{\frac{1}{2s}}P_s\ast T\big)\big\rvert \psi_Q(\ox)\dd\ox \\
		&\leq c_{n,s}\|\partial_t^{\frac{1}{2s}}P_s\ast T \|_{\text{BMO}_{p_s}(\psi_Q)},
	\end{align*}
	being the latter a weighted $s$-parabolic BMO space defined via the regular weight $\psi_Q$. The latter clearly belongs to the Muckenhoupt class $A_\infty$ and therefore, by a well-known classical result \cite{MuWh} which admits an extension to this in the current $s$-parabolic setting (just take into account $s$-parabolic cubes in the definition of the bounded mean oscillation space), we have the continuous inclusion $\text{BMO}_{p_s}\hookrightarrow \text{BMO}_{p_s}(\psi_Q)$ and we deduce \eqref{eq3.2.6}. Therefore,
	\begin{align*}
		\int_R|\text{II}(\ox)|\dd\ox &= \int_Q|\varphi(x,t)|\big\rvert \partial_t^{\frac{1}{2s}}P_s\ast T(x,t)-m_{\psi_Q}\big(\partial_t^{\frac{1}{2s}}P_s\ast T \big) \big\rvert \dd\ox\\
		&\lesssim \ell(Q)^{n+2s}+\int_Q\big\rvert \partial_t^{\frac{1}{2s}}P_s\ast T(x,t)-m_{\psi_Q}\big(\partial_t^{\frac{1}{2s}}P_s\ast T \big) \big\rvert \\
		&\leq \ell(Q)^{n+2s}+\ell(Q)^{n+2s}\|\partial_t^{\frac{1}{2s}}P_s\ast T\|_{\ast,p_s}\lesssim \ell(R)^{n+2s}.
	\end{align*}
	Let us now deal with $\text{III}$:
	\begin{align*}
		\text{III}(\ox)&:=\int_{\mathbb{R}}\frac{\big( \varphi(x,r)-\varphi(x,t) \big)\big( P_s\ast T(x,r)-P_s\ast T(x,t) \big)}{|r-t|^{1+\frac{1}{2s}}}\dd r\\
		&=\int_{|r-t|\leq \ell(Q)^{2s}}\frac{\big( \varphi(x,r)-\varphi(x,t) \big)\big( P_s\ast T(x,r)-P_s\ast T(x,t) \big)}{|r-t|^{1+\frac{1}{2s}}}\dd r\\
		&\hspace{2cm}+\int_{|r-t|>\ell(Q)^{2s}}\frac{\big( \varphi(x,r)-\varphi(x,t) \big)\big( P_s\ast T(x,r)-P_s\ast T(x,t) \big)}{|r-t|^{1+\frac{1}{2s}}}\dd r\\
		&=:\text{III}_1(\ox)+\text{III}_2(\ox).
	\end{align*}
	By the $\text{Lip}_{\frac{1}{2s},t}$ property of $P_s\ast T$ (Theorem \ref{thm3.1.1}) and $\|\partial_t\varphi\|_\infty\leq \ell(Q)^{-2s}$,
	\begin{align*}
		|\text{III}_1(\ox)|\lesssim \int_{|r-t|\leq \ell(Q)^{2s}}\frac{\ell(Q)^{-2s}|r-t||r-t|^{\frac{1}{2s}}}{|r-t|^{1+\frac{1}{2s}}}\dd r\lesssim 1,
	\end{align*}
	so that $\int_R |\text{III}_1(\ox)|\dd \ox\lesssim \ell(R)^{n+2s}$. For $\text{III}_2$ we have
	\begin{align*}
		\text{III}_2(\ox)&=\int_{|r-t|>\ell(Q)^{2s}}\frac{P_s\ast T(x,r)-P_s\ast T(x,t)}{|r-t|^{1+\frac{1}{2s}}}\varphi(x,r)\dd r\\
		&\hspace{0.5cm}-\varphi(x,t)\int_{|r-t|>\ell(Q)^{2s}}\frac{P_s\ast T(x,r)-P_s\ast T(x,t)}{|r-t|^{1+\frac{1}{2s}}}\dd r=:\text{III}_{21}(\ox)-\text{III}_{22}(\ox).
	\end{align*}
	Using again the Lipschitz property and the fact that $|\varphi(x,\cdot)|\leq \chi_{I_Q}$ we obtain
	\begin{equation*}
		|\text{III}_{21}(\ox)|\lesssim \int_{|r-t|>\ell(Q)^{2s}}\frac{|\varphi(x,r)|}{|r-t|}\dd r<\frac{1}{\ell(Q)^{2s}}\int |\varphi(x,r)|\dd r \lesssim 1,
	\end{equation*}
	so that we also have $\int_R |\text{III}_{21}(\ox)|\dd \ox\lesssim \ell(R)^{n+2s}$. Then, the only remaining term to study is $|-\text{III}_{22}+\text{IV}|$, $\text{III}_{22}$ and $\text{IV}$ being both supported on $Q$. So to conclude the proof it suffices to show for $\ox\in Q$:
	\begin{equation}
		\label{eq3.2.7}
		\bigg\rvert m_{\psi_Q}\big( \partial_t^{\frac{1}{2s}}P_s\ast T \big)-\int_{|r-t|>\ell(Q)^{2s}}\frac{P_s\ast T(x,r)-P_s\ast T(x,t)}{|r-t|^{1+\frac{1}{2s}}}\dd r \bigg\rvert \lesssim 1.
	\end{equation}
	Let us turn our attention to $m_{\psi_Q}( \partial_t^{\frac{1}{2s}}P_s\ast T )$ and begin by noticing
	\begin{align*}
		\partial_t^{\frac{1}{2s}}P_s\ast T (y,u) &= \int_{|r-u|\leq \ell(Q)^{2s}}\frac{P_s\ast T(y,r)-P_s\ast T (y,u)}{|r-u|^{1+\frac{1}{2s}}}\dd r\\
		&\hspace{1cm}+\int_{|r-u|> \ell(Q)^{2s}}\frac{P_s\ast T(y,r)-P_s\ast T (y,u)}{|r-u|^{1+\frac{1}{2s}}}\dd r =:F_1(\oy)+F_2(\oy).
	\end{align*}
	Observe that the kernel
	\begin{equation*}
		K_y(r,u):=\chi_{|r-u|\leq \ell(Q)^{2s}}\frac{P_s\ast T(y,r)-P_s\ast T(y,u)}{|r-u|^{1+\frac{1}{2s}}}
	\end{equation*}
	is antisymmetric, and therefore
	\begin{align*}
		m_{\psi_Q}F_1 &= \frac{1}{\|\psi_Q\|_1}\iiint K_y(r,u)\psi_Q(y,u)\dd r \dd y \dd u\\
        &=-\frac{1}{\|\psi_Q\|_1}\iiint K_y(r,u)\psi_Q(y,r)\dd u \dd y \dd r\\
		&=\frac{1}{2\|\psi_Q\|_1}\iiint K_y(r,u)\big(\psi_Q(y,u)-\psi_Q(y,r)\big)\dd r \dd y \dd u.
	\end{align*}
	Then,
	\begin{align*}
		|&m_{\psi_Q}F_1|\\
		&\leq \frac{1}{2\|\psi_Q\|_1}\iint \bigg(\int_{|r-u|\leq \ell(Q)^{2s}}\frac{|P_s\ast T(y,r)-P_s\ast T(y,u)|}{|r-u|^{1+\frac{1}{2s}}}\big\rvert \psi_Q(y,u)-\psi_Q(y,r) \big\rvert\dd u\bigg) \dd y \dd r\\
		&\lesssim \frac{1}{\ell(Q)^{n+2s}}\int_{2Q_1}\int_{2^{2s}I_Q}\int_{|r-u|\leq \ell(Q)^{2s}}\frac{|r-u|^{\frac{1}{2s}}|r-u|}{|r-u|^{1+\frac{1}{2s}}\ell(Q)^{2s}}\dd u\dd y \dd r\lesssim 1. 
	\end{align*}
	So returning to \eqref{eq3.2.7} we are left to show
	\begin{align*}
		\bigg\rvert m_{\psi_Q}F_2-\int_{|r-t|>\ell(Q)^{2s}}\frac{P_s\ast T(x,r)-P_s\ast T(x,t)}{|r-t|^{1+\frac{1}{2s}}}\dd r \bigg\rvert \lesssim 1, \qquad (x,t)\in Q,
	\end{align*}
	that in turn is implied by the inequality
	\begin{align*}
		\bigg\rvert F_2(\oy)&- \int_{|r-t|>\ell(Q)^{2s}}\frac{P_s\ast T(x,r)-P_s\ast T(x,t)}{|r-t|^{1+\frac{1}{2s}}}\dd r \bigg\rvert =|F_2(\oy)-F_2(\ox)|\lesssim 1,
	\end{align*}
    where $(y,u)\in 2Q$. Write $A_t:=\{r\,:\, |r-t|>\ell(Q)^{2s}\}$ and  $A_u:=\{r\,:\, |r-u|>\ell(Q)^{2s}\}$, so that
	\begin{align*}
		|F_2&(\oy)-F_2(\ox)|\\
        &\leq \int_{A_u\setminus{A_t}} \frac{P_s\ast T(y,r)-P_s\ast T(y,u)}{|r-u|^{1+\frac{1}{2s}}}\dd r+\int_{A_t\setminus{A_u}} \frac{P_s\ast T(x,r)-P_s\ast T(x,t)}{|r-t|^{1+\frac{1}{2s}}}\dd r\\
		&\hspace{0.75cm}+\int_{A_u\cap A_t}\bigg\rvert \frac{P_s\ast T(x,r)-P_s\ast T(x,t)}{|r-t|^{1+\frac{1}{2s}}} - \frac{P_s\ast T(y,r)-P_s\ast T(y,u)}{|r-u|^{1+\frac{1}{2s}}} \bigg\rvert\dd r\\
		&=:\text{I}'+\text{II}'+\text{III}'.
	\end{align*}
	Using the $\text{Lip}_{\frac{1}{2s},t}$ property of $P_s\ast T$ and that $|r-u|\approx \ell(Q)^{2s}$ in $A_u\setminus A_t$, and $|r-t|\approx \ell(Q)^{2s}$ in $A_t\setminus A_u$, it is easily checked $\text{I}'+\text{II}'\lesssim 1$. Concerning $\text{III}'$ we split
	\begin{align*}
		\text{III}'\leq \int_{A_u\cap A_t}\bigg\rvert &\frac{1}{|r-t|^{1+\frac{1}{2s}}} - \frac{1}{|r-u|^{1+\frac{1}{2s}}} \bigg\rvert|P_s\ast T(x,r)-P_s\ast T (x,t)|\dd r\\
		&+\int_{A_u\cap A_t}\frac{|P_s\ast T(x,r)-P_s\ast T(x,t)-P_s\ast T(y,r)+P_s\ast T(y,u)|}{|r-u|^{1+\frac{1}{2s}}}\dd r\\
		&=:\text{III}'_{1}+\text{III}'_{2}.
	\end{align*}
	For $\text{III}'_{1}$ notice that, in the domain of integration, by the mean value theorem we have
	\begin{equation*}
		\bigg\rvert \frac{1}{|r-t|^{1+\frac{1}{2s}}} - \frac{1}{|r-u|^{1+\frac{1}{2s}}} \bigg\rvert\lesssim \frac{|t-u|}{|r-t|^{2+\frac{1}{2s}}}.
	\end{equation*}
	Therefore, by the $\text{Lip}_{\frac{1}{2s},t}$ property of $P_s\ast T$,
	\begin{equation*}
		\text{III}'_{1}\lesssim \int_{|r-t|>\ell(Q)^{2s}}\frac{|t-u|}{|r-t|^{2+\frac{1}{2s}}}|r-t|^{\frac{1}{2s}}\dd r \lesssim 1.
	\end{equation*}
	Regarding $\text{III}'_{2}$, apply the $(1,\frac{1}{2s})$-Lipschitz property of $P_s\ast T$ and obtain
	\begin{align*}
		\text{III}'_{2}&\leq \int_{|r-u|>\ell(Q)^{2s}} \frac{|P_s\ast T(x,r)-P_s\ast T(y,r)|+|P_s\ast T(y,u)+P_s\ast T(x,t)|}{|r-u|^{1+\frac{1}{2s}}}\dd r\\
		&\lesssim \int_{|r-u|>\ell(Q)^{2s}}\frac{\ell(Q)}{|r-u|^{1+\frac{1}{2s}}}\dd r \lesssim 1.
	\end{align*}
	This finally shows $|F_2(\oy)-F_2(\ox)|\lesssim 1$, for all $\ox\in Q$ and $\oy\in 2Q$ and we are done.
\end{proof}
\begin{lem}
	\label{lem3.2.6}
	Let $Q\subset \mathbb{R}^{n+1}$ be an $s$-parabolic cube and let $T$ be a distribution supported in $\mathbb{R}^{n+1}\setminus{4Q}$ with upper $s$-parabolic growth of degree $n+1$ and such that $\|\nabla_xP_s\ast T\|_\infty\leq 1$ and $\|P_s\ast T\|_{\text{\normalfont{Lip}}_{\frac{1}{2s},t}}\leq 1$. Then,
	\begin{equation*}
		\int_Q \big\rvert  \partial_t^{\frac{1}{2s}}P_s\ast T(\ox) -\big( \partial_t^{\frac{1}{2s}}P_s\ast T \big)_Q \big\rvert \dd \ox \lesssim \ell(Q)^{n+2s}.
	\end{equation*}
\end{lem}
\begin{proof}
	Let us fix $Q$ $s$-parabolic cube. To obtain the result it suffices to verify
	\begin{equation*}
		\big\rvert  \partial_t^{\frac{1}{2s}}P_s\ast T(\ox) - \partial_t^{\frac{1}{2s}}P_s\ast T(\oy) \big\rvert \lesssim 1, \qquad \text{for }\; \ox,\oy \in Q.
	\end{equation*}
	To be precise, in order to avoid possible $t$-differentiability obstacles regarding the kernel $P_s$, we resort a standard regularization process: take $\psi$ test function supported on the unit $s$-parabolic ball $B(0,1)$ such that $\int \psi = 1$ and set $\psi_\varepsilon:=\varepsilon^{-n-2s}\psi(\cdot/\varepsilon)$ and the regularized kernel $P_s^{\varepsilon}:=\psi_\varepsilon \ast P_s$. As it already mentioned in \cite{MPr}, the previous kernel satisfies the same growth estimates as $P_s$ (see \cite[Theorem 2.2]{HeMPr} or \cite[Lemma 2.2]{MPr}).  If we prove
	\begin{equation*}
		\big\rvert  \partial_t^{\frac{1}{2s}}P_s^{\varepsilon}\ast T(\ox) - \partial_t^{\frac{1}{2s}}P_s^\varepsilon\ast T(\oy) \big\rvert \lesssim 1, \qquad \text{for }\; \ox,\oy \in Q.
	\end{equation*}
	uniformly on $\varepsilon$ we will be done, since making $\varepsilon\to 0$ would allow us to recover the original expression for almost every point. Moreover, we shall also assume that
	\begin{enumerate}
		\item[\textit{i}.] $\ox=(x,t)$ and $\oy=(y,t)$,
		\item[\textit{ii}.] or that $\ox =(x,t)$ and $\oy = (x,u)$.
	\end{enumerate}
	Let first tackle case \textit{i}. We compute:
	\begin{align*}
		\big\rvert  \partial_t^{\frac{1}{2s}}&P_s^{\varepsilon}\ast T(x,t) - \partial_t^{\frac{1}{2s}}P_s^{\varepsilon}\ast T(y,t) \big\rvert\\
		&=\bigg\rvert \int_{\mathbb{R}} \frac{P_s^{\varepsilon}\ast T(x,r)-P_s^{\varepsilon}\ast T(x,t)}{|r-t|^{1+\frac{1}{2s}}}\dd r -  \int_{\mathbb{R}} \frac{P_s^{\varepsilon}\ast T(y,r)-P_s^{\varepsilon}\ast T(y,t)}{|r-t|^{1+\frac{1}{2s}}}\dd r\bigg\rvert\\
		&\leq \int_{|r-t|\leq 2^{2s}\ell(Q)^{2s}}\frac{|P_s^{\varepsilon}\ast T(x,r)-P_s^{\varepsilon}\ast T(x,t)|}{|r-t|^{1+\frac{1}{2s}}}\dd r\\
		&\hspace{0.75cm}+\int_{|r-t|\leq 2^{2s}\ell(Q)^{2s}}\frac{|P_s^{\varepsilon}\ast T(y,r)-P_s^{\varepsilon}\ast T(y,t)|}{|r-t|^{1+\frac{1}{2s}}}\dd r\\
		&\hspace{0.75cm}+\int_{|r-t|>2^{2s}\ell(Q)^{2s}}\frac{|P_s^{\varepsilon}\ast T(x,r)-P_s^{\varepsilon}\ast T(x,t)-P_s^{\varepsilon}\ast T(y,r)+P_s^{\varepsilon}\ast T(y,t)|}{|r-u|^{1+\frac{1}{2s}}}\dd r\\
		&=:\text{I}+\text{II}+\text{III}.
	\end{align*}
	Let us estimate $\text{I}$. For $r,t$ such that $|r-t|\leq 2^{2s}\ell(Q)^{2s}$, we write
	\begin{equation*}
		|P_s^{\varepsilon}\ast T(x,r)-P_s^{\varepsilon}\ast T(x,t)|\leq |r-t|\|\partial_tP_s^\varepsilon\ast T\|_{\infty, 3Q}.
	\end{equation*}
	We claim that
	\begin{equation}
		\label{eq3.2.8}
		\|\partial_tP_s^\varepsilon\ast T\|_{\infty, 3Q}\lesssim \frac{1}{\ell(Q)^{2s-1}}.
	\end{equation}
	If this holds,
	\begin{equation*}
		\text{I}\lesssim \int_{|r-t|\leq 2^{2s}\ell(Q)^{2s}} \frac{|r-t|}{\ell(Q)^{2s-1}|r-t|^{1+\frac{1}{2s}}}\dd r \lesssim \frac{\big( \ell(Q)^{2s}\big)^{1-\frac{1}{2s}}}{\ell(Q)^{2s-1}}=1.
	\end{equation*}
	Analogously, interchanging the roles of $\ox$ and $\oy$, we deduce $\text{II}\lesssim 1$. Regarding $\text{III}$,
	\begin{align*}
		\text{III}\leq \int_{|r-t|>2^{2s}\ell(Q)^{2s}} &\frac{|P_s^{\varepsilon}\ast T(x,r)-P_s^{\varepsilon}\ast T(y,r)|}{|r-u|^{1+\frac{1}{2s}}}\dd r\\
		&+\int_{|r-t|>2^{2s}\ell(Q)^{2s}} \frac{|P_s^{\varepsilon}\ast T(x,t)-P_s^{\varepsilon}\ast T(y,t)|}{|r-u|^{1+\frac{1}{2s}}}\dd r.
	\end{align*}
	Then $|P_s^{\varepsilon}\ast T(x,r)-P_s^{\varepsilon}\ast T(y,r)|\leq \|\nabla_xP_s^{\varepsilon}\ast T\|_{\infty}|x-y|\leq \|\psi_\varepsilon\|_1\|\nabla_xP_s\ast T\|_{\infty}|x-y|\lesssim \ell(Q)$. The same holds replacing $(x,r)$ and $(y,r)$ by $(x,t)$ and $(y,t)$. Hence,
	\begin{equation*}
		\text{III}\lesssim \int_{|r-t|>2^{2s}\ell(Q)^{2s}}\frac{\ell(Q)}{|r-t|^{1+\frac{1}{2s}}}\dd r \lesssim \frac{\ell(Q)}{\big( \ell(Q)^{2s}\big)^{\frac{1}{2s}}}=1.
	\end{equation*}
	So once \eqref{eq3.2.8} is proved, case \textit{i}) is done. To prove it, we split $\mathbb{R}^{n+1}\setminus{4Q}$ into $s$-parabolic annuli $A_k:=2^{k+1}Q\setminus 2^k Q$ and consider $\pazocal{C}^\infty$ functions $\widetilde{\chi}_k$ such that $\chi_{A_k}\leq \widetilde{\chi}_k \leq \chi_{\frac{11}{10}A_k}$, as well as $\|\nabla_x \widetilde{\chi}_k\|_\infty \lesssim (2^{k}\ell(Q))^{-1}$, $\|\partial_t \widetilde{\chi}_k\|_\infty \lesssim (2^k\ell(Q))^{-2s}$ and $\sum_{k\geq 2}\widetilde{\chi}_k = 1$ on $\mathbb{R}^{n+1}\setminus{4Q}$.
    
	Let us fix $\oz = (z,v)\in 3Q$ and observe
	\begin{equation*}
		|\partial_tP_s^{\varepsilon}\ast T (\oz)|\leq \sum_{k\geq 2} |\partial_tP_s^{\varepsilon}\ast \widetilde{\chi}_kT (\oz)|.
	\end{equation*}
	We prove $|\partial_tP_s^{\varepsilon}\ast \widetilde{\chi}_kT (\oz)|\lesssim (2^k\ell(Q))^{-2s+1}$, and with it we will be done. Write
	\begin{equation*}
		|\partial_tP_s^{\varepsilon}\ast \widetilde{\chi}_kT (\oz)|=|\langle T, \widetilde{\chi}_k\partial_tP_s^{\varepsilon}(\oz-\cdot)\rangle|=:|\langle T, \phi_{k,\oz}^{\varepsilon}\rangle|.
	\end{equation*}
	We study $\|\nabla_x \phi^{\varepsilon}_{k,\oz}\|_\infty$ and $\|\partial_t \phi^{\varepsilon}_{k,\oz}\|_\infty$ in order to apply the upper $s$-parabolic growth of $T$. On the one hand we have, for each $\ox=(x,t)\in A_k$ with $t \neq v$, by \cite[Lemma 2.2]{MPr} and \cite[Theorem 2.2]{HeMPr},
	\begin{align*}
		|\nabla_x\phi_{k,\oz}^{\varepsilon}(\ox)|&\leq |\nabla_x\widetilde{\chi}_k\cdot\partial_tP_s^{\varepsilon}(\oz-\ox) |+|\widetilde{\chi}_k\cdot\nabla_x\partial_tP_s^{\varepsilon}(\oz-\ox)|\\
		&\lesssim (2^k\ell(Q))^{-1}\frac{1}{|\ox-\oz|_{p_s}^{n+2s}}+\frac{1}{|\ox-\oz|_{p_s}^{n+2s+1}}\lesssim (2^k\ell(Q))^{-(n+2s+1)},
	\end{align*}
	and then $\|\nabla_x\phi_{k,\oz}^{\varepsilon}\|\lesssim (2^k\ell(Q))^{-(n+2s+1)}$. Similarly, for each $\ox=(x,t)\in A_k$ with $t\neq v$,
	\begin{align*}
		|\partial_t\phi_{k,\oz}^{\varepsilon}(\ox)|&\leq |\partial_t\widetilde{\chi}_k\cdot\partial_tP_s^{\varepsilon}(\oz-\ox) |+|\widetilde{\chi}_k\cdot\partial_t^2P_s^{\varepsilon}(\oz-\ox)|\\
		&\lesssim (2^k\ell(Q))^{-2s}\frac{1}{|\ox-\oz|_{p_s}^{n+2s}}+\frac{1}{|\ox-\oz|_{p_s}^{n+4s}}\lesssim (2^k\ell(Q))^{-(n+4s)},
	\end{align*}
	where the bound for $\partial_t^2P_s^{\varepsilon}$ can be argued with same arguments to those of \cite[Lemma 2.2]{MPr}, for example. Then, $\|\partial_t\phi_{k,\oz}^{\varepsilon}\|\lesssim (2^k\ell(Q))^{-(n+4s)}$ and with this we deduce that $(2^k\ell(Q))^{n+2s}\phi_{k,\oz}^{\varepsilon}$ is a function to which we can apply Theorem \ref{thm3.2.1}. This way,
	\begin{equation*}
		|\partial_tP_s^{\varepsilon}\ast \widetilde{\chi}_kT (\oz)|=:|\langle T, \phi_{k,\oz}^{\varepsilon}\rangle|\lesssim \frac{1}{(2^k\ell(Q))^{n+2s}}(2^k\ell(Q))^{n+1} = (2^k\ell(Q))^{-2s+1},
	\end{equation*}
	that is what we wanted to prove, and this ends case \textit{i}.
    
	We move on to case \textit{ii}, that is, $\ox =(x,t)$ and $\oy = (x,u)$. We proceed as in \textit{i} and write
	\begin{align*}
		\big\rvert  \partial_t^{\frac{1}{2s}}&P_s^{\varepsilon}\ast T(x,t) - \partial_t^{\frac{1}{2s}}P_s^{\varepsilon}\ast T(x,u) \big\rvert\\
		&=\bigg\rvert \int_{\mathbb{R}} \frac{P_s^{\varepsilon}\ast T(x,r)-P_s^{\varepsilon}\ast T(x,t)}{|r-t|^{1+\frac{1}{2s}}}\dd r -  \int_{\mathbb{R}} \frac{P_s^{\varepsilon}\ast T(x,r)-P_s^{\varepsilon}\ast T(x,u)}{|r-u|^{1+\frac{1}{2s}}}\dd r\bigg\rvert\\
		&\leq \int_{|r-t|\leq 2^{2s}\ell(Q)^{2s}}\frac{|P_s^{\varepsilon}\ast T(x,r)-P_s^{\varepsilon}\ast T(x,t)|}{|r-t|^{1+\frac{1}{2s}}}\dd r\\
		&\hspace{0.75cm}+\int_{|r-t|\leq 2^{2s}\ell(Q)^{2s}}\frac{|P_s^{\varepsilon}\ast T(x,r)-P_s^{\varepsilon}\ast T(x,u)|}{|r-u|^{1+\frac{1}{2s}}}\dd r\\
		&\hspace{0.75cm}+\int_{|r-t|>2^{2s}\ell(Q)^{2s}} \bigg\rvert \frac{P_s\ast T(x,r)-P_s\ast T(x,t)}{|r-t|^{1+\frac{1}{2s}}} - \frac{P_s\ast T(x,r)-P_s\ast T(x,u)}{|r-u|^{1+\frac{1}{2s}}} \bigg\rvert \dd r\\
		&=:\text{I}'+\text{II}'+\text{III}'.
	\end{align*}
	$\text{I}'$ and $\text{II}'$ can be dealt with as $\text{I}$ and $\text{II}$ appearing in case \textit{i}. Then $\text{I}'+\text{II}'\lesssim 1$. Concerning $\text{III}'$ we have, by the $\text{Lip}_{\frac{1}{2s},t}$ property of $P_s\ast T$ (and thus of $P_s^\varepsilon\ast T$, since $\psi_\varepsilon$ integrates 1),
	\begin{align*}
		\text{III}'&\leq \int_{|r-t|>2^{2s}\ell(Q)^{2s}}\bigg\rvert \frac{1}{|r-t|^{1+\frac{1}{2s}}} - \frac{1}{|r-u|^{1+\frac{1}{2s}}} \bigg\rvert |P_s\ast T(x,r)-P_s\ast T(x,t)|\dd r\\
		&\hspace{3.5cm}+\int_{|r-t|>2^{2s}\ell(Q)^{2s}}\frac{1}{|r-u|^{1+\frac{1}{2s}}}|P_s\ast T(x,t)-P_s\ast T(x,u)|\dd r\\
		&\lesssim \int_{|r-t|>2^{2s}\ell(Q)^{2s}}\frac{\ell(Q)^{2s}}{|r-t|^{2+\frac{1}{2s}}}|r-t|^{\frac{1}{2s}}\dd r + \int_{|r-t|>2^{2s}\ell(Q)^{2s}}\frac{|t-u|^{\frac{1}{2s}}}{|r-u|^{1+\frac{1}{2s}}}\dd r\lesssim 1.
	\end{align*}
\end{proof}
Applying the above results we are able to deduce the following final lemma:
\begin{lem}
	\label{lem3.2.7}
	Let $T$ be a distribution in $\mathbb{R}^{n+1}$ satisfying
	\begin{equation*}
		\|\nabla_x P_s \ast T\|_{\infty}\leq 1, \qquad \|\partial_t^{\frac{1}{2s}}P_s\ast T\|_{\ast,p_s}\leq 1.
	\end{equation*}
	Let $Q$ be an $s$-parabolic cube and $\varphi$ admissible function for $Q$. Then,
	\begin{equation*}
		\|\partial_t^{\frac{1}{2s}}P_s\ast \varphi T\|_{\ast,p_s}\lesssim 1.
	\end{equation*}
\end{lem}
\begin{proof}
	We fix $\widetilde{Q}$ any $s$-parabolic cube and prove that there exists $c_{\widetilde{Q}}$ constant so that
	\begin{equation*}
		\int_{\widetilde{Q}}\big\rvert \partial_t^{\frac{1}{2s}}P_s\ast \varphi T (\ox)-c_{\widetilde{Q}}\big\rvert \dd \ox \lesssim \ell(\widetilde{Q})^{n+2s}.
	\end{equation*}
	To do so, pick a $\pazocal{C}^\infty$ bump function $\phi_{5\widetilde{Q}}$ with $\chi_{5\widetilde{Q}}\leq \phi_{5\widetilde{Q}}\leq \chi_{6\widetilde{Q}}$ and satisfying
	\begin{equation*}
		\|\nabla_x \phi_{5\widetilde{Q}}\|_{\infty} \lesssim \ell(\widetilde{Q})^{-1}, \qquad \|\Delta \phi_{5\widetilde{Q}}\|_{\infty} \lesssim \ell(\widetilde{Q})^{-2}, \qquad \|\partial_t \phi_{5\widetilde{Q}}\|_{\infty} \lesssim \ell(\widetilde{Q})^{-2s}.
	\end{equation*}
	We also write $\phi_{5\widetilde{Q}^c}:=1-\phi_{5\widetilde{Q}}$. Then we split
	\begin{align*}
		\int_{\widetilde{Q}}\big\rvert &\partial_t^{\frac{1}{2s}}P_s\ast \varphi T (\ox)-c_{\widetilde{Q}}\big\rvert \dd \ox\\
		&\leq \int_{\widetilde{Q}}\big\rvert \partial_t^{\frac{1}{2s}}P_s\ast \phi_{5\widetilde{Q}}\varphi T (\ox)\big\rvert \dd \ox + \int_{\widetilde{Q}}\big\rvert \partial_t^{\frac{1}{2s}}P_s\ast \phi_{5\widetilde{Q}^c}\varphi T (\ox)-c_{\widetilde{Q}}\big\rvert \dd \ox =:\text{I}+\text{II}.
	\end{align*}
	Let us estimate $\text{II}$ applying Lemma \ref{lem3.2.6}. Notice that $\text{supp}(\phi_{5\widetilde{Q}^c}\varphi T)\subset \overline{(5\widetilde{Q})^c}\cap Q$. We claim that
	\begin{equation}
		\label{eq3.2.9}
		\|\nabla_x P_s\ast \phi_{5\widetilde{Q}^c}\varphi T\|_\infty\lesssim 1 \qquad \text{and} \qquad \|P_s\ast \phi_{5\widetilde{Q}^c}\varphi T\|_{\text{Lip}_{\frac{1}{2s},t}}\lesssim 1.
	\end{equation}
	To check this, we write $P_s\ast \phi_{5\widetilde{Q}^c}\varphi T=P_s\ast \varphi T - P_s\ast \phi_{5\widetilde{Q}}\varphi T$ and since $\varphi$ is admissible for $Q$ we already have
	\begin{equation*}
		\|\nabla_x P_s\ast \varphi T\|_\infty\lesssim 1 \qquad \text{and} \qquad \|P_s\ast \varphi T\|_{\text{Lip}_{\frac{1}{2s},t}}\lesssim 1,
	\end{equation*}
	by Lemmas \ref{lem3.2.2} and \ref{lem3.2.4}. Let us observe that, if $\ell(\widetilde{Q})\leq \ell(Q)$, there exists some constant that makes $\phi_{5\widetilde{Q}}\varphi$ admissible for $5\widetilde{Q}$. On the other hand, if $\ell(\widetilde{Q})>\ell(Q)$, then there is another constant making $\phi_{5\widetilde{Q}}\varphi$ admissible for $Q$. So in any case, also by Lemmas \ref{lem3.2.2} and \ref{lem3.2.4}, we have
	\begin{equation*}
		\|\nabla_x P_s\ast \phi_{5\widetilde{Q}}\varphi T\|_\infty\lesssim 1 \qquad \text{and} \qquad \|P_s\ast \phi_{5\widetilde{Q}}\varphi T\|_{\text{Lip}_{\frac{1}{2s},t}}\lesssim 1,
	\end{equation*}
	and \eqref{eq3.2.9} follows. With this in mind and the fact that $\phi_{5\widetilde{Q}}\varphi T$ has upper $s$-parabolic growth of degree $n+1$ (use that $\phi_{5\widetilde{Q}}\varphi T$ is either admissible for $Q$ or $\widetilde{Q}$ and apply Theorem \ref{thm3.2.1}), we choose
	\begin{equation*}
		c_{\widetilde{Q}}:=\big( \partial_t^{\frac{1}{2s}}P_s\ast \phi_{5\widetilde{Q}}\varphi T \big)_{\widetilde{Q}}
	\end{equation*}
	and apply Lemma \ref{lem3.2.6} to obtain $\text{II}\lesssim \ell(\widetilde{Q})^{n+2s}$.
    
	To study $\text{I}$, we shall assume $Q\cap 6\widetilde{Q}\neq \varnothing$. Now, if $\ell(\widetilde{Q})\leq \ell(Q)$, we have that for some constant $\phi_{5\widetilde{Q}}\varphi$ is admissible for $5\widetilde{Q}$. Applying Lemma \ref{lem3.2.5} with both cubes of its statement equal to $5\widetilde{Q}$, we get
	\begin{equation*}
		\text{I} \leq \int_{5\widetilde{Q}}\big\rvert \partial_t^{\frac{1}{2s}}P_s\ast \phi_{5\widetilde{Q}}\varphi T (\ox)\big\rvert \dd \ox \lesssim \ell(\widetilde{Q})^{n+2s}.
	\end{equation*}
	If $\ell(\widetilde{Q})>\ell(Q)$, since $Q\cap 6\widetilde{Q}\neq \varnothing$ we deduce $Q\subset 8\widetilde{Q}$ and hence $\phi_{5\widetilde{Q}}\varphi$ is admissible for $Q$ for some constant. Applying again Lemma \ref{lem3.2.5} now for the cubes $Q$ and $8\widetilde{Q}$, we have
	\begin{equation*}
		\text{I} \leq \int_{8\widetilde{Q}}\big\rvert \partial_t^{\frac{1}{2s}}P_s\ast \phi_{5\widetilde{Q}}\varphi T (\ox)\big\rvert \dd \ox \lesssim \ell(\widetilde{Q})^{n+2s},
	\end{equation*}
	and we are done.
\end{proof}
Combining Lemmas \ref{lem3.2.2} and \ref{lem3.2.7} we are able to finally state the main theorem of this subsection:
\begin{thm}
	\label{thm3.2.8}
	Let $T$ be a distribution in $\mathbb{R}^{n+1}$ satisfying
	\begin{equation*}
		\|\nabla_x P_s \ast T\|_{\infty}\leq 1, \qquad \|\partial_t^{\frac{1}{2s}}P_s\ast T\|_{\ast,p_s}\leq 1.
	\end{equation*}
	Let $Q$ be an $s$-parabolic cube and $\varphi$ admissible function for $Q$. Then, $\varphi T$ is an admissible distribution, up to a constant depending on $n$ and $s$, for $\Gamma_{\Theta^s}(Q)$.
\end{thm}

\section{Removable singularities. The critical dimension of \mathinhead{\Gamma_{\Theta^s}}{}}
\label{sec3.3}
The main reason to prove the localization result of the previous section is to obtain a characterization of removable sets for solutions of the $\Theta^s$-equation satisfying a $(1, \frac{1}{2s})$-Lipschitz property.
\begin{thm}
	\label{thm3.3.1}
	A compact set $E\subset \mathbb{R}^{n+1}$ is removable for Lipschitz $s$-caloric functions if and only if $\Gamma_{\Theta^s}(E)=0$.
\end{thm}
\begin{proof}
	We consider $s<1$, since the case $s=1$ is covered in \cite[Theorem 5.3]{MPrT}.
	Let $E\subset \mathbb{R}^{n+1}$ be compact and assume that is removable for Lipschitz $s$-caloric functions. Let $T$ be admissible for $\Gamma_{\Theta^s}(E)$ and define $f:=P_s\ast T$, so that $\|\nabla_xf\|_{\infty}<\infty$, $\|\partial_t^{\frac{1}{2s}}f\|_{\ast,p_s}<\infty$ and $\Theta^s f = 0$ on $\mathbb{R}^{n+1}\setminus{E}$. By hypothesis $\Theta^s f = 0$ in $\mathbb{R}^{n+1}$ so $T\equiv 0$, and then $\Gamma_{\Theta^s}(E)=0$.
    
	Assume now that $E$ is not removable for Lipschitz $s$-caloric functions. Then, there exists $\Omega\supset E$ open set and $f:\mathbb{R}^{n+1}\to \mathbb{R}$ with
	\begin{equation*}
		\|\nabla_x f\|_{\infty}<\infty, \qquad \|\partial_t^{\frac{1}{2s}}f\|_{\ast,p_s}<\infty
	\end{equation*}
	such that $\Theta^s f = 0$ on $\Omega\setminus{E}$, but $\Theta^s f \neq 0$ on $\Omega$. Define the distribution
	\begin{equation*}
		T:=\frac{\Theta^s f}{\|\nabla_x f\|_{\infty}+\|\partial_t^{\frac{1}{2s}}f\|_{\ast,p_s}},
	\end{equation*}
	that is such that $\|\nabla_x P_s \ast T\|_{\infty}\leq 1$, $\|\partial_t^{\frac{1}{2s}} P_s \ast T\|_{\ast,p_s}\leq 1$ and $\text{supp}(T)\subset E \cup \Omega^c$. Since $T\neq 0$ in $\Omega$, there exists $Q$ $s$-parabolic cube with $4Q\subset \Omega$ so that $T\neq 0$ in $Q$. Observe that $Q\cap E \neq \varnothing$. Then, by definition, there is $\varphi$ test function supported on $Q$ with $\langle T, \varphi \rangle>0$. Consider
	\begin{equation*}
		\widetilde{\varphi}:=\frac{\varphi}{\|\varphi\|_\infty + \ell(Q)\|\nabla_x\varphi\|_\infty+\ell(Q)^{2s}\|\partial_t\varphi\|_\infty+\ell(Q)^2\|\Delta_x \varphi\|_\infty},
	\end{equation*}
	so that $\widetilde{\varphi}$ is admissible for $Q$. Apply Theorem \ref{thm3.2.8} to deduce that $\widetilde{\varphi} T$ is admissible for $\Gamma_{\Theta^s}(E)$ (up to a constant) and therefore
	\begin{align*}
		\Gamma_{\Theta^s}(E) &\gtrsim \frac{1}{\|\varphi\|_\infty + \ell(Q)\|\nabla_x\varphi\|_\infty+\ell(Q)^{2s}\|\partial_t\varphi\|_\infty+\ell(Q)^2\|\Delta_x \varphi\|_\infty} |\langle \varphi T, 1 \rangle|>0.
	\end{align*}
\end{proof}
Let us prove now that, given $E\subset \mathbb{R}^{n+1}$, its removability for Lipschitz $s$-caloric functions will be tightly related to its $s$-parabolic Hausdorff dimension:
\begin{thm}
	\label{thm3.3.2}
	For every compact set $E\subset \mathbb{R}^{n+1}$ the following hold:
	\begin{enumerate}
		\item[1.] $\Gamma_{\Theta^s}(E)\leq C\,\pazocal{H}_{\infty,p_s}^{n+1}(E)$, for some constant $C(n,s)>0$.
		\item[2.] If $\text{\normalfont{dim}}_{\pazocal{H}_{p_s}}(E)>n+1$, then $\Gamma_{\Theta^s}(E)>0$.
	\end{enumerate}
	Therefore, the critical $s$-parabolic Hausdorff dimension of $\Gamma_{\Theta^s}$ in $\mathbb{R}^{n+1}$ is $n+1$.
\end{thm}
\begin{proof}
	Again, let us restrict ourselves to $s<1$, since the case $s=1$ is already covered in \cite[Lemma 5.1]{MPrT}.

    For the first estimate, proceed by fixing $\varepsilon>0$ and $\{A_k\}_k$ a collection of sets in $\mathbb{R}^{n+1}$ that cover $E$ such that
		\begin{equation*}
			\sum_{k=1}^\infty \text{diam}_{p_s}(A_k)^{n+1}\leq \pazocal{H}^{n+1}_{\infty,p_s}(E)+\varepsilon.
		\end{equation*}
		Now, for each $k$ let $Q_k$ an open $s$-parabolic cube centered at some point $a_k\in A_k$ with side length $\ell(Q_k)=\text{diam}_{p_s}(A_k)$, so that $E\subset \bigcup_kQ_k$. Apply the compactness of $E$ and \cite[Lemma 3.1]{HPo} to consider $\{\varphi_k\}_{k=1}^N$ a collection of smooth functions satisfying, for each $k$: $0\leq \varphi_k\leq 1$, $\text{supp}(\varphi_k)\subset 2Q_k$, $\sum_{k=1}^N\varphi_k = 1$ in $\bigcup_{k=1}^N Q_k$ and also $\|\nabla_x\varphi_k\|_\infty\leq \ell(2Q_k)^{-1}$, $\|\partial_t\varphi_k\|\leq \ell(2Q_k)^{-2s}$. Hence, by Theorem \ref{thm3.2.1}, if $T$ is any distribution admissible for $\Gamma_{\Theta^s}(E)$,
		\begin{align*}
			|\langle T, 1 \rangle |=\bigg\rvert \sum_{k=1}^N \langle T, \varphi_k \rangle \bigg\rvert  \lesssim \sum_{k=1}^N \ell(2Q_k)^{n+1} \simeq \sum_{k=1}^N\text{diam}_{p_s}(A_k)^{n+1}\leq \pazocal{H}^{n+1}_{\infty,p_s}(E)+\varepsilon.
		\end{align*}
		Since this holds for any $T$ and $\varepsilon>0$ can be arbitrarily small, the desired estimate follows.
    
	In order to prove \textit{2} we apply an $s$-parabolic version of Frostman's lemma. Let us name $d:=\text{\normalfont{dim}}_{\pazocal{H}_{p_s}}(E)$ and assume $0<\pazocal{H}_{p_s}^d(E)<\infty$ without loss of generality. Indeed, if $\pazocal{H}_{p_s}^d(E)=\infty$, apply an $s$-parabolic version of \cite[Theorem 4.10]{F} to construct a compact set $\widetilde{E}\subset \mathbb{R}^{n+1}$ with $\widetilde{E}\subset E$ and $0<\pazocal{H}_{p_s}^{d}(\widetilde{E})<\infty$. On the other hand, if $\pazocal{H}_{p_s}^d(E)=0$ apply the same reasoning with $d':=(d+n+1)/2$. In any case, by Frostman's lemma we shall then consider a non trivial positive measure $\mu$ with $\text{supp}(\mu)\subset E$ satisfying $\mu(B(\ox,r))\leq r^{d}$, for all $\ox\in\mathbb{R}^{n+1}$ and all $r>0$. Observe that if we prove
	\begin{equation*}
		\big\|\nabla_x P_s\ast \mu\big\|_{\infty}\lesssim 1\hspace{0.5cm} \text{and}\hspace{0.5cm} \big\|\partial^{\frac{1}{2s}}_t P_s\ast \mu\big\|_{\ast,p_s}\lesssim 1,
	\end{equation*}
	we will be done, since this would imply $\Gamma_{\Theta^s}(E)\geq \Gamma_{\Theta^s}(\widetilde{E}) \gtrsim \langle \mu, 1 \rangle = \mu(\widetilde{E})>0$. The bound for $\partial^{\frac{1}{2s}}_t P_s\ast \mu$ follows directly from \cite[Lemma 4.2]{HeMPr} with $\beta:=\frac{1}{2s}$. To prove that of $\nabla_x P_s\ast \mu$, use \cite[Lemma 2.2]{MPr} to deduce that for any $\ox\in\mathbb{R}^{n+1}$,
	\begin{equation*}
		|\nabla_x P_s\ast \mu(\ox)|\lesssim \int_{E}\frac{\text{d}\mu(\oz)}{|\ox-\oy|_{p_s}^{n+1}}\lesssim \text{diam}_{p_s}(E)^{d-(n+1)}\lesssim 1,
	\end{equation*}
	where we have split the previous domain into annuli and used that $\mu$ presents upper $s$-parabolic growth of degree $d>n+1$.
\end{proof}
We proceed by providing a result regarding the capacity of subsets of $\pazocal{H}^{n+1}_{p_s}$-positive measure of regular $\text{Lip}(1,\frac{1}{2s})$ graphs. To proceed, let us introduce the following operator: for a given $\mu$, a real compactly supported Borel regular measure with upper $s$-parabolic growth of degree $n+1$, we define the operator $\pazocal{P}_{\mu}^{s}$ acting on elements of $L^1_{\text{loc}}(\mu)$ as
\begin{equation*}
	\pazocal{P}_{\mu}^{s}f(\ox):=\int_{\mathbb{R}^{n+1}}\nabla_x P_s(\ox-\oy)f(\oy)\text{d}\mu(\oy), \hspace{0.5cm} \ox\notin \text{supp}(\mu).
\end{equation*}
In the particular case in which $f$ is the constant function $1$ we write
\begin{equation*}
	\pazocal{P}^s\mu(\ox):=\pazocal{P}^s_{\mu}1 (\ox).
\end{equation*}
It is clear that the previous expression is defined pointwise on $\mathbb{R}^{n+1}\setminus{\text{supp}(\ox)}$. We also introduce the truncated version of $\pazocal{P}^s_{\mu}$,
\begin{equation*}
	\pazocal{P}^s_{\mu,\varepsilon}f(\ox):=\int_{|\ox-\oy|>\varepsilon}\nabla_x P_s(\ox-\oy)f(\oy)\text{d}\mu(\oy),\hspace{0.5cm} \ox\in \mathbb{R}^{n+1}, \; \varepsilon>0.
\end{equation*}
For a given $1\leq p \leq \infty$, we will say that $\pazocal{P}^s_\mu f$ belongs to $L^p(\mu)$ if the $L^p(\mu)$-norm of the truncations $\|\pazocal{P}^s_{\mu,\varepsilon}f\|_{L^p(\mu)}$ is  uniformly bounded on $\varepsilon$, and we write
	\begin{equation*}
		\|\pazocal{P}^s_\mu f\|_{L^p(\mu)}:=\sup_{\varepsilon>0}\|\pazocal{P}^s_{\mu,\varepsilon}f\|_{L^p(\mu)}
	\end{equation*}
	We will say that the operator $\pazocal{P}^s_\mu$ is bounded on $L^p(\mu)$ if the operators $\pazocal{P}^s_{\mu,\varepsilon}$ are bounded on $L^p(\mu)$ uniformly on $\varepsilon$, and we equally set
	\begin{equation*}
		\|\pazocal{P}^s_\mu\|_{L^p(\mu)\to L^p(\mu)}:=\sup_{\varepsilon>0}\|\pazocal{P}^s_{\mu,\varepsilon}\|_{L^p(\mu)\to L^p(\mu)}.
	\end{equation*}
We also introduce the notation used for maximal operators
\begin{align*}
\pazocal{P}^s_{\ast,\mu}f(\ox)&:=\sup_{\varepsilon>0}|\pazocal{P}^s_{\mu,\varepsilon}f(\ox)|.
\end{align*}
The same definitions apply for the conjugate operator $\pazocal{P}^{s,\ast}$, associated with the conjugate kernel $P_s^\ast (\ox):=P_s(-\ox)$.

We denote by $\Sigma_{n+1}^s(E)$ the collection of positive Borel measures supported on $E$ with upper $s$-parabolic growth of degree $n+1$ with constant 1 and define the auxiliary capacity:
\begin{equation}
    \label{new_eq3.3.1}
	\widetilde{\Gamma}_{\Theta^s,+}(E):=\sup \mu(E),
\end{equation}
where the supremum is taken over all measures $\mu\in\Sigma_{n+1}^s(E)$ such that
\begin{equation*}
	\|\pazocal{P}^{s}\mu\|_\infty \leq 1, \qquad \|\pazocal{P}^{s,\ast}\mu\|_\infty \leq 1.
\end{equation*}
We aim at proving the following result, analogous to \cite[Theorem 5.5]{MPrT}:
\begin{thm}
	\label{thm3.3.3}
	Let $E\subset \mathbb{R}^{n+1}$ be a compact set. Then, for each $s\in(1,2,1]$,
	\begin{equation*}
		\widetilde{\Gamma}_{\Theta^s,+}(E)\lesssim \Gamma_{\Theta^s,+}(E) \approx \sup\big\{ \mu(E)\,:\, \mu\in\Sigma_{n+1}^s(E), \, \|\pazocal{P}^{s}\mu\|_\infty \leq 1 \big\}.
	\end{equation*}
	Moreover,
	\begin{equation*}
		\widetilde{\Gamma}_{\Theta^s,+}(E) \approx \sup\big\{ \mu(E)\,:\, \mu\in\Sigma_{n+1}^s(E), \, \|\pazocal{P}^{s}_{\mu}\|_{L^2(\mu)\to L^2(\mu)} \leq 1 \big\}.
	\end{equation*}
\end{thm}
Previous to that, let us verify the following lemma, which follows from the growth estimates proved for the kernel $\nabla_xP_s$ in \cite[Theorem 2.2]{HeMPr} and analogous arguments to those in \cite[Lemma 5.4]{MatPa}, for example:

\begin{lem}
	\label{lem3.3.4}
	Let $\mu$ be a real Borel measure with compact support and upper $s$-parabolic growth of degree $n+1$ with $\|\pazocal{P}^{s}\mu\|_{\infty}\leq 1$. Then, there is $\kappa>0$, constant depending on $n$ and $s$, so that
	\begin{equation*}
		\pazocal{P}^{s}_{\ast}\mu (\ox)\leq \kappa, \hspace{0.5cm} \forall\, \ox\in \mathbb{R}^{n+1}.
	\end{equation*}
\end{lem}

\begin{proof}
	Let $|\mu|$ denote the variation of $\mu$. which corresponds to $|\mu|=\mu^++\mu^-$, where $\mu^+,\mu^-$ are the positive and negative variations of $\mu$ respectively, defined as the set functions
	\begin{align*}
		\mu^+(B)&=\sup\{\mu(A)\;:\; A\in\pazocal{B}(\mathbb{R}^{n+1}), A\subset B\},\\
		\mu^-(B)&=-\inf\{\mu(A)\;:\; A\in\pazocal{B}(\mathbb{R}^{n+1}), A\subset B\}.
	\end{align*}
	It is known that $|\mu|$ is a positive measure \cite[Theorem 6.2]{Ru} with $\mu\ll|\mu|$. In fact, there exists an $L^1_{\text{loc}}(\mathbb{R}^{n+1})$ function $g$ with $g(\ox)=\pm 1$ such that $\mu=g|\mu|$ \cite[Theorem 6.12]{Ru}. It is clear that $|\mu|$ still satisfies the same upper $s$-parabolic growth condition as $\mu$.
    
	Let us fix $0<\varepsilon<1$ and $\ox=(x,t)\in \mathbb{R}^{n+1}$ and notice that for some $0<\delta<2s-1$,
	\begin{align*}
		\int_{B(\ox,\varepsilon/4)}\bigg(\int_{B(\ox,\varepsilon)}&\frac{\text{d}|\mu|(\oy)}{|\oz-\oy|_{p_s}^{n+1}}\bigg)\dd \oz \leq \int_{B(\ox,\varepsilon)}\bigg(\int_{B(\oy,2\varepsilon)} \frac{\dd \oz}{|\oz-\oy|^{n+1}_{p_s}} \bigg)\text{d}|\mu|(\oy)\\
		&\hspace{-2.75cm}\leq \int_{B(\ox,\varepsilon)}\bigg(\int_{B_1(y,2\varepsilon)} \frac{\dd z}{|z-y|^{n-\delta}} \int_{-2^{2s}\varepsilon^{2s}}^{2^{2s}\varepsilon^{2s}}\frac{du}{u^{\frac{1+\delta}{2s}}}\bigg)\text{d}|\mu|(\oy) \lesssim \varepsilon^{\delta}\varepsilon^{2s-1-\delta}|\mu|(B(\ox,\varepsilon))\lesssim \varepsilon^{n+2s}.
	\end{align*}
	By the first estimate of \cite[Theorem 2.2]{HeMPr} this implies, in particular, that we can find some $\oz\in B(\ox,\varepsilon/4)$ such that $|\pazocal{P}^{s}\mu(\oz)|\leq \|\pazocal{P}^{s}\mu\|_{\infty} \leq 1$ and satisfying
	\begin{equation*}
		|\pazocal{P}_{\mu}^{s} \chi_{B(\ox,\varepsilon)} (\oz)|\leq C_2,
	\end{equation*}
	for some positive constant $C_2$. Therefore, we obtain
	\begin{align*}
		\big\rvert \pazocal{P}_{\varepsilon}^{s}\mu(\ox)- \pazocal{P}^{s}\mu(\oz)\big\rvert &\leq \big\rvert \pazocal{P}_{\varepsilon}^{s}\mu(\ox)-\pazocal{P}_{\varepsilon}^{s}\mu(\oz) \big\rvert +C_2\\
		&\leq \int_{|\ox-\oy|_{p_s}>\frac{\varepsilon}{2}}\big\rvert \nabla_xP_s(\ox-\oy)-\nabla_xP_s(\oz-\oy) \big\rvert \text{d}|\mu|(\oy) +C_2   
	\end{align*}
	Applying the last estimate of \cite[Theorem 2.2]{HeMPr} we get
	\begin{align*}
		\int_{|\ox-\oy|_{p_s}>\frac{\varepsilon}{2}}\big\rvert \nabla_xP_s(\ox-\oy)-\nabla_xP_s(\oz-\oy) \big\rvert \text{d}|\mu|(\oy)\lesssim \varepsilon \int_{|\ox-\oy|_{p_s}>\frac{\varepsilon}{2}}\frac{\dd |\mu|(\oy)}{|\ox-\oy|_{p_s}^{n+2}}.
	\end{align*}
	Splitting the domain of integration into $s$-parabolic annuli:
	\begin{align*}
		A_j:=\big\{\oy\,:\, 2^{j-1}\varepsilon \leq |\ox-\oy|_{p_s} \leq 2^{j}\varepsilon \big\}, \quad j\geq 0,
	\end{align*}
	and using the growth of $|\mu|$ we obtain, for some positive constant $C_1$,
	\begin{align*}
		\varepsilon\int_{|\ox-\oy|_{p_s}>\frac{\varepsilon}{2}}\frac{\dd |\mu|(\oy)}{|\ox-\oy|_{p_s}^{n+2}}\leq \varepsilon\sum_{j=0}^\infty \int_{A_j} \frac{\dd |\mu|(\oy)}{|\ox-\oy|_{p_s}^{n+2}}\lesssim \varepsilon \sum_{j=0}^\infty \frac{(2^j\varepsilon)^{n+1}}{(2^{j-1}\varepsilon)^{n+2}}\leq C_1.
	\end{align*}
	Thus, we have established the bound
	\begin{align*}
		|\pazocal{P}_{\varepsilon}^{s}\mu(\ox)-\pazocal{P}^{s}\mu(\oz)|&\leq C_1+C_2,
	\end{align*}
	so setting $\kappa:=C_1+C_2+\|\pazocal{P}^{s}\mu\|_{\infty}$ and using that by hypothesis $\|\pazocal{P}^{s}\mu\|_{\infty}\leq 1$, we get the desired inequality.
\end{proof}

\begin{proof}[Proof of Theorem \ref{thm3.3.3}]
	Denote
	\begin{align*}
		\Gamma_1& := \sup\{\mu(E):\mu\in\Sigma_{n+1}^s(E),\,\|\pazocal{P}^{s}\mu\|_{\infty}\leq1\}, \\
		\Gamma_2 & :=\sup\{\mu(E):\mu\in\Sigma_{n+1}^s(E),\,\|\pazocal{P}^{s}_{\mu}\|_{L^2(\mu)\to L^2(\mu)}\leq1\}.
	\end{align*}
	It is clear that $\widetilde{\Gamma}_{\Theta^s,+}(E)\leq \Gamma_1$. It is also clear that Theorem \ref{thm3.2.1} implies $\Gamma_1\gtrsim \Gamma_{\Theta^s,+}(E)$, and that the converse estimate follows from \cite[Lemma 4.2]{HeMPr} applied with $\beta:=\frac{1}{2s}$.
	
	To prove that $\widetilde{\Gamma}_{\Theta^s,+}(E)\approx \Gamma_2$ we argue as in \cite[Theorem 5.5]{MPrT}. Take $\mu\in\Sigma_{n+1}^s(E)$ satisfying
	$\widetilde{\Gamma}_{\Theta^s,+}(E)\leq 2\mu(E)$ and $\|\pazocal{P}^{s}\mu\|_{\infty}\leq1$, $\|\pazocal{P}^{s,\ast}\mu\|_{\infty}\leq1$. By Lemma \ref{lem3.3.4} we get, uniformly on $\varepsilon>0$,
	\begin{equation}
		\label{eq3.3.1}
		\|\pazocal{P}_{\varepsilon}^{s}\mu\|_{L^\infty(\mu)} \lesssim 1,\qquad \|\pazocal{P}_{\varepsilon}^{s,\ast}\mu\|_{L^\infty(\mu)} \lesssim 1,
	\end{equation}
	The boundedness of $\pazocal{P}_{\mu}^{s}$ in $L^2(\mu)$ will follow from a 
	$Tb$ theorem of Hyt\"onen and Martikainen \cite[Theorem 2.3]{HyMa} for non-doubling measures in geometrically doubling spaces, such as the $s$-parabolic space. We shall apply the previous theorem with $b=1$. Taking into account \eqref{eq3.3.1}, to ensure that $\pazocal{P}_{\mu}^s$ is bounded in $L^2(\mu)$, by \cite[Theorem 2.3]{HyMa} it is enough to verify that the weak boundedness property holds for $s$-parabolic balls with thin boundary. An $s$-parabolic ball of radius $r_B$ has $A$-thin boundary if
	\begin{equation}\label{eq3.3.3}
		\mu\big\{\ox:\text{dist}_{p_s}(\ox,\partial B)\le \lambda r_B\big\}\le A\,\lambda\mu(2B)\quad \mbox{ $\forall\lambda\in(0,1)$.}
	\end{equation} 
	Hence, we need to prove that, for some fixed $A>0$ and any $B\subset\mathbb{R}^{n+1}$ $s$-parabolic ball with $A$-thin boundary,
	\begin{equation}
		\label{eq3.3.2}
		|\langle \pazocal{P}^{s}_{\mu,\varepsilon}\chi_B, \chi_B\rangle| \le C\mu(2B),\quad\mbox{ uniformly on $\varepsilon>0$}.
	\end{equation}
	
	To prove \eqref{eq3.3.2}, consider a smooth function $\varphi$ compactly supported on $2B$ with $\varphi\equiv1$ on $B$ and write
	$$
	|\langle \pazocal{P}^{s}_{\mu,\varepsilon}\chi_B,\chi_B\rangle|\le \int_B|\pazocal{P}^{s}_{\mu,\varepsilon}\varphi|\dd \mu+\int_B|\pazocal{P}_{\mu,\varepsilon}^s(\varphi-\chi_B)|\dd\mu.
	$$
	Since $\mu$ presents $n+1$ upper $s$-parabolic growth and $\|\pazocal{P}^{s}\mu\|_{\infty  }\leq1$, by \cite[Lemma 4.2]{HeMPr} and the localization Theorem \ref{thm3.2.8} we have $\|\pazocal{P}^{s}_{\mu}\varphi\|_{\infty}\leq1$, which in turn implies that $\|\pazocal{P}^{s}_{\mu,\varepsilon}\varphi\|_{\infty}\leq1$ uniformly on $\varepsilon>0$. Therefore, the first integral on the right side of the above inequality
	is bounded by $C\mu(B)$. To estimate the second term we will use that $B$ has a thin boundary and the growth estimates of \cite[Theorem 2.2]{HeMPr}. We compute:
	\begin{align*}
		\int_B|\pazocal{P}^{s}_{\mu,\varepsilon}(\varphi-\chi_B)|\dd\mu &\lesssim \int_{2B\setminus B}\int_B\frac{\dd \mu(y)}{|x-y|_{p_s}^{n+1}}\dd \mu(x)\\
		&
		\le \sum_{j\ge 0}\int_{ \big\{\frac{r_B}{2^{j}}\leq \text{dist}_{p_s}(x,\partial B) \leq \frac{r_B}{2^{j-1}}\big\}\setminus{B}}\int_B\frac{\dd \mu(y)}{|x-y|_{p_s}^{n+1}}\dd \mu(x).
	\end{align*}
	Given $j$  and $x\notin B$ with $2^{-j}r_B\leq\text{dist}_{p_s}(x,\partial B)\leq2^{-j+1}r_B$, since $\mu\in \Sigma_{n+1}^s(E)$ we get 
	\begin{align*}
		\int_B\frac{d\mu(y)}{|x-y|_{p_s}^{n+1}}&\leq\sum_{k=-1}^{k=j}\int_{\frac{r_B}{2^{k+1}}\leq|x-y|_{p_s}\leq \frac{r_B}{2^k}} 
		\frac{\dd\mu(y)}{|x-y|_{p_s}^{n+1}}\lesssim \sum_{k=-1}^{k=j}\frac{\mu(B(x,2^{-k}r_B))}{(2^{-k}r_B)^{n+1}}\lesssim j+2. 
	\end{align*}
	Therefore, by \eqref{eq3.3.3}
	\begin{align*}
		\int_B|\pazocal{P}^{s}_{\mu,\varepsilon}(\varphi-\chi_B)|\dd\mu&\lesssim 
		\sum_{j\ge 0}(j+2)\,\mu(\{x: 2^{-j}r_B\leq\text{dist}_{p_s}(x,\partial B)\leq 2^{-j+1}r(B)\})\\
		&\lesssim \sum_{j\ge 0}\frac{j+2}{2^j}\mu(2B) \lesssim \mu(2B).
	\end{align*}
	So \eqref{eq3.3.2} holds and  
	$\pazocal{P}_{\mu}^{s}$  satisfies
	$\|\pazocal{P}_{\mu}^{s}\|_{L^2(\mu)\to L^2(\mu)}\lesssim1$. Therefore $\Gamma_2\gtrsim \widetilde{\Gamma}_{\Theta^s,+}(E)$.
    
	To prove the upper estimate, take $\mu\in\Sigma_{n+1}^s(E)$ with  $\|\pazocal{P}^{s}_{\mu}\|_{L^2(\mu)\to L^2(\mu)}\leq 1$
	and $\Gamma_2\leq 2\mu(E)$. The $L^2(\mu)$ boundedness of $\pazocal{P}^{s}_{\mu}$ ensures that $\pazocal{P}^{s}$ and $\pazocal{P}^{s,\ast}$ are bounded from the space of finite signed measures $M(\mathbb{R}^{n+1})$ to $L^{1,\infty}(\mu)$. In other words, there exists a constant $C>0$ so that, given any $\nu\in M(\mathbb{R}^{n+1})$, for any $\varepsilon>0$ and $\lambda>0$,
	\begin{equation*}
		\mu\big(\big\{x\in\mathbb{R}^{n+1}:|\pazocal{P}_{\varepsilon}^{s} \nu(x)|>\lambda \big\}\big) \leq C\,\frac{\|\nu\|}\lambda,
	\end{equation*} 
	and the same replacing $\pazocal{P}_{\varepsilon}^{s}$ by $\pazocal{P}^{s,\ast}_\varepsilon$.
	The reader can consult a proof of the latter in \cite[Theorem 2.16]{T} (to apply the previous arguments, it is used that the Besicovitch covering theorem with respect to $s$-parabolic balls is valid. Otherwise, the reader may also consult \cite[Theorem 5.1]{NTrVo}.) From this point on, by a well known dualization of these estimates (essentially due to Davie and Øksendal \cite{DaØ}, see \cite[Ch.VII, Theorem 23]{C} for a precise statement) and an application of Cotlar's inequality (see \cite[Theorem 2.18]{T}, for example), we deduce the existence of a function $h:E\to [0,1]$ so that
	\begin{equation*}
		\mu(E)\leq C\,\int h\dd\mu,\quad \; \|\pazocal{P}_{\mu}^sh\|_{\infty}\leq 1,\quad \; \|\pazocal{P}_{\mu}^{s,\ast}h\|_{\infty}\leq 1.
	\end{equation*}
	Therefore, $\widetilde{\Gamma}_{\Theta^s,+}(E)\geq \int h\dd\mu\approx \mu(E)\approx \Gamma_2,$ and we are done with the proof.
\end{proof}

Applying the previous result and arguing as in \cite[Example 5.6]{MPrT}, we get that any subset of $\pazocal{H}_{p_s}^{n+1}$-positive measure of the graph of a $\text{Lip}(1,\frac{1}{2s})$ function is not removable. It is not clear, however, if for $s<1$ such an object exists.

\subsection{Existence of removable sets with positive \mathinhead{\pazocal{H}_{p_s}^{n+1}}{} measure}
\label{subsec3.3.1}

We would like to carry out a similar study of the capacity introduced in \cite[\textsection 4]{MPr} for  typical corner-like Cantor sets, but in the current $s$-Lipschitz context. Let us remark that in \cite{MPr} the critical dimension of the capacity in $\mathbb{R}^{n+1}$ is $n$, and that the ambient space is endowed with the usual Euclidean distance.

The first question to ask is if it suffices to consider the same corner-like Cantor set of the aforementioned reference, but now consisting on the intersection of successive families of $s$-parabolic cubes. If the reader is familiar with \cite[\textsection 6]{MPrT}, he or she might anticipate that the answer to the previous question is negative. Let us motivate why this is the case.

Recall that for a given sequence $\lambda:=(\lambda_k)_k, \, 0< \lambda_k < 1/2$, we define its associated Cantor set $E\subset \mathbb{R}^{n+1}$ by the following algorithm: set $Q^0:=[0,1]^{n+1}$ the unit cube of $\mathbb{R}^{n+1}$ and consider $2^{n+1}$ disjoint (Euclidean) cubes inside $Q^0$ of side length $\ell_1:=\lambda_1$, with sides parallel to the coordinate axes and such that each cube contains a vertex of $Q^0$. Continue this same process now for each of the $2^{n+1}$ cubes from the previous step, but now using a contraction factor $\lambda_2$. That is, we end up with $2^{2(n+1)}$ cubes with side length $\ell_2:=\lambda_1\lambda_2$. Proceeding inductively we have that at the $k$-th step of the iteration we encounter $2^{k(n+1)}$ cubes, that we denote $Q_j^k$ for $1\leq j \leq 2^{k(n+1)}$, with side length $\ell_k:=\prod_{j=1}^k \lambda_j$. We will refer to them as cubes of the $k$-\textit{th generation}. We define
\begin{equation*}
	E_k=E(\lambda_1,\ldots,\lambda_k):=\bigcup_{j=1}^{2^{k(n+1)}}Q_j^k,
\end{equation*}
and from the latter we obtain the Cantor set associated with $\lambda$,
\begin{equation*}
	E=E(\lambda):=\bigcap_{k=1}^\infty E_k.
\end{equation*}
If we chose $\lambda_j=2^{-(n+1)/n}$ for every $j$, we would recover the particular Cantor set presented in \cite[\textsection 5]{MPr}. The previous choice is so particular that ensures
\begin{equation*}
	0<\pazocal{H}^n(E) \simeq \pazocal{H}_{p_{1/2}}^n(E) <\infty.
\end{equation*}
If $\#(E_k)$ is the number of cubes of $E_k$, the above property followed, in essence, from
\begin{equation}
	\label{eq2.2.1}
	\#(E_k)\, \ell_k^{n} = 2^{k(n+1)}\, \ell_k^{n}  = 1, \;\; \forall k \geq 1.
\end{equation}
If we were to obtain such critical value of $\lambda_j$ in the $s$-parabolic setting, taking into account the critical dimension of $\Gamma_{\Theta^s}$, it should be such that
\begin{equation*}
	0<\pazocal{H}^{n+1}_{p_s}(E)<\infty.
\end{equation*}
So if we directly consider the analog of the previous corner-like Cantor set, but made up of $s$-parabolic cubes, we should rewrite \eqref{eq2.2.1} as
\begin{equation*}
	2^{k(n+1)}\cdot \ell_k^{n+1} = 1, \;\; \forall k \geq 1,
\end{equation*}
meaning that the corresponding critical value of $\lambda_j$ has to be $1/2$, that is not admissible. Another reason that suggests that working with an $s$-parabolic version of the corner-like Cantor set could not be the best choice, is that it becomes too \textit{small} in an $s$-parabolic Hausdorff-dimensional sense. Indeed, if we assume that there is $\tau_0$ so that $\lambda_j\leq \tau_0 <1/2,\, \forall j$, for any fixed $0<\varepsilon\ll 1$ we may choose a generation $k\gg 1$ so that
\begin{equation*}
	\pazocal{H}^{n+1}_{\varepsilon,p_s}(E)\leq \pazocal{H}^{n+1}_{\varepsilon,p_s}(E_k) \lesssim 2^{k(n+1)}\ell_{k}^{n+1}\leq (2\tau_0)^{k(n+1)}\xrightarrow[k\to\infty]{} 0,
\end{equation*}
that implies $\Gamma_{\Theta^s}(E)=0$, by Theorem \ref{thm3.3.2}.

Hence, it is clear that in order to obtain a potentially non-removable Cantor set $E$, one has to \textit{enlarge} it. One way to do it (motivated by \cite[\textsection 6]{MPrT}) is as follows: let us fix $s\in(1/2,1]$ and choose what we call the \textit{non-self-intersection} parameter $d\in\mathbb{Z}_+$, the minimum integer $d=d(s)\geq 2$ satisfying
\begin{equation*}
	s>\frac{\log_d(d+1)}{2}, \qquad \text{that is} \qquad d+1 < d^{2s}.
\end{equation*}
Let $Q^0:=[0,1]^{n+1}$ be the unit cube of $\mathbb{R}^{n+1}$ and consider $(d+1)d^n$ disjoint $s$-parabolic cubes $Q_i^1, \, 1\leq i \leq (d+1)d^n$, contained in $Q^0$, with sides parallel to the coordinate axes, side length $0<\lambda_1<1/d$, and disposed as follows: first, we consider the first $n$ intervals of the cartesian product $Q^0:=[0,1]^{n+1}$ (that is, those contained in spatial directions) and we divide, each one, into $d$ equal subintervals $I_{1},\ldots,I_d$. For each subinterval $I_j$, we contain another one $J_j$ of length $\lambda_1$. Now, we distribute $J_1,\ldots,J_d$ in an equispaced way, fixing $J_1$ to start at 0 and $J_d$ to end at 1. More precisely, if for each interval $[0,1]$ we name
\begin{equation*}
	l_d:=\frac{1-d\lambda_1}{d-1}, \qquad J_j:=\big[(j-1)(\lambda_1+l_d), j\lambda_1+(j-1)l_d\big], \quad j=1,\ldots, d,
\end{equation*}
we keep the following union of $d$ closed disjoint intervals of length $\lambda_1$
\begin{align*}
	T_d:=\bigcup_{j=1}^d J_j.
\end{align*}
Finally, for the remaining temporal interval $[0,1]$, we do the same splitting but in $d+1$ intervals of length $\lambda_1^{2s}$. That is, we name
\begin{equation*}
	\widetilde{l}_{d}:=\frac{1-(d+1)\lambda_1^{2s}}{d}, \qquad \widetilde{J}_j:=\big[(j-1)(\lambda_1^{2s}+\widetilde{l}_{d}), j\lambda_1^{2s}+(j-1)\widetilde{l}_d\big], \quad j=1,\ldots, d+1.
\end{equation*}
Now we keep the union of $d+1$ closed disjoint intervals of length $\lambda_1^{2s}$
\begin{align*}
	\widetilde{T}_d:=\bigcup_{j=1}^{d+1} \widetilde{J}_j.
\end{align*}
From the above, we define the first generation of the Cantor set as $E_{1,p_s}:=(T_d)^n\times \widetilde{T}_d$, that is conformed by $(d+1)d^n$ disjoint $s$-parabolic cubes (see Figure \ref{Q0Q1}). This procedure continues inductively, i.e. the next generation $E_{2,p_s}$ will be the family of $(d+1)^2d^{2n}$ disjoint $s$-parabolic cubes of side length $\ell_2 := \lambda_1\lambda_2$, $\lambda_2<1/d$, obtained from applying the previous construction to each of the cubes of $E_{1,p_s}$. More generally, the $k$-th generation $E_{k,p_s}$ will be formed by $(d+1)^kd^{nk}$ disjoint $s$-parabolic cubes with side length $\ell_k := \lambda_1\cdots \lambda_k$, $\lambda_k<1/d$, and with locations determined by the above iterative process. We name such cubes $Q^k_j$, with $j=1,\ldots, (d+1)^kd^{nk}$. The resulting $s$-\textit{parabolic} Cantor set is
\begin{equation}
	\label{eq3.3.5}
	E_{p_s} = E_{p_s}(\lambda) :=  \bigcap_{k=1}^{\infty} E_{k,p_s}.
\end{equation}
\begin{figure}[t]
	\centering
	\begin{subfigure}[b]{0.44\textwidth}
		\centering
		\includegraphics[width=\textwidth]{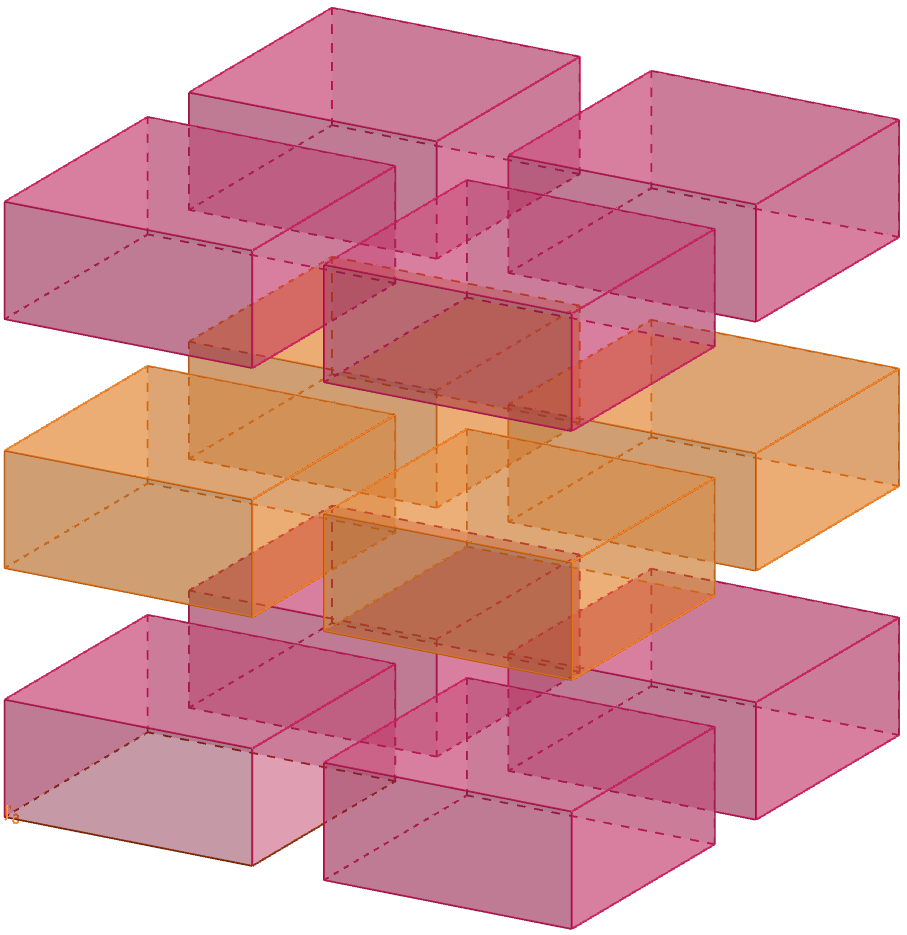}
		\caption{Depiction of $E_{1,p_1}$ with $d=2$}
		\label{Q0}
	\end{subfigure}
	\qquad
	\begin{subfigure}[b]{0.44\textwidth}
		\centering
		\includegraphics[width=\textwidth]{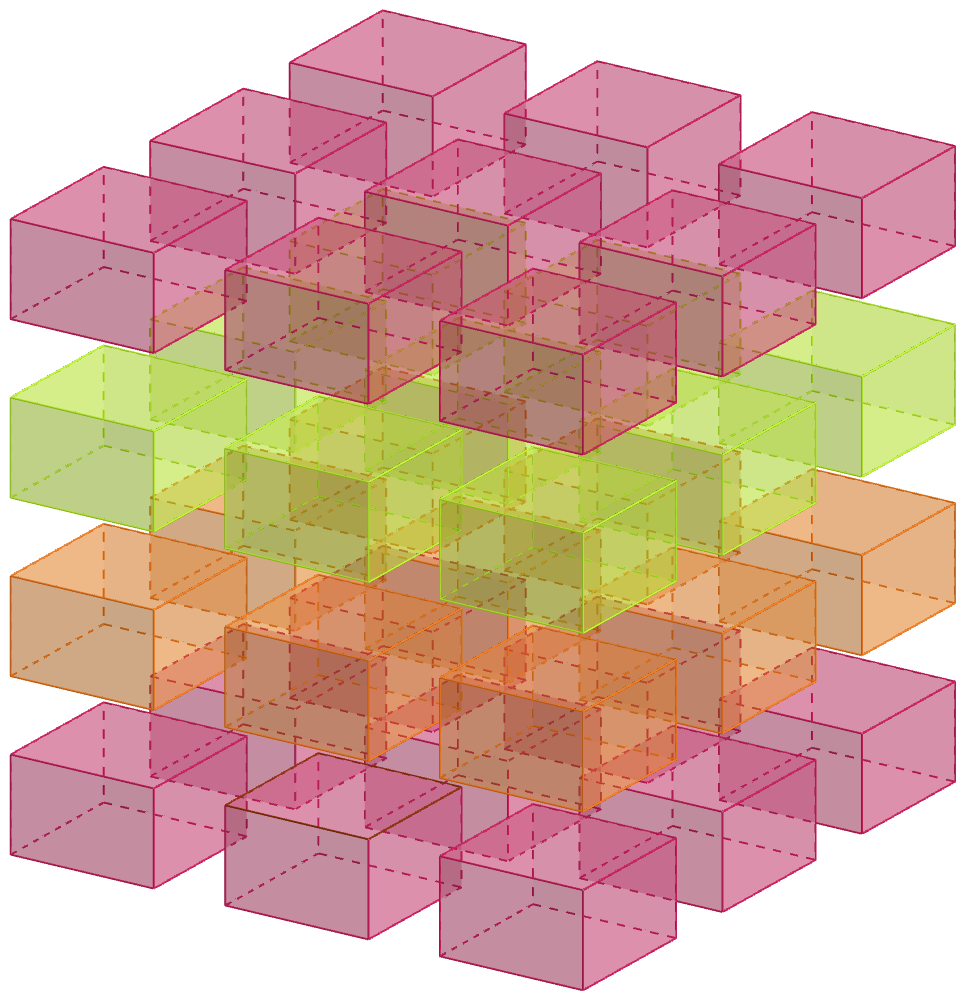}
		\caption{Depiction of $E_{1,p_{2/3}}$ with $d=3$.}
		\label{Q1}
	\end{subfigure}
	\caption{First iterates involved in the construction of $E_{p_1}$ and $E_{p_{2/3}}$ in $\mathbb{R}^3$. For $s=1$ we have chosen $\lambda_1 := 12^{-1/3}$, and for $s=2/3$ we have chosen $\lambda_1:=1/4$.}
	\label{Q0Q1}
\end{figure}

This way, the critical value becomes $((d+1)d^n)^{-1/(n+1)}<1/d$. For instance, if
\begin{equation}
	\label{eq3.3.6}
	\lambda_j := ((d+1)d^n)^{-1/(n+1)}, \qquad \text{for every }\, j,
\end{equation}
then
\begin{equation*}
	\#(E_{k,p_s})\cdot \ell_k^{n+1} =  (d+1)^kd^{nk}\cdot \ell_k^{n+1} = 1, \; \;\forall k\geq 1.
\end{equation*}
Using the previous fact one can deduce $\pazocal{H}^{n+1}_{p_s}(E_{p_s})>0$. Indeed, consider the probability measure $\mu$ defined on $E_{p_s}$ such that for each generation $k$, $\mu(Q_j^k):=(d+1)^{-k}d^{-nk},\, 1 \leq j \leq (d+1)^kd^{nk}$. Let $Q$ be any $s$-parabolic cube, that we may assume to be contained in $Q^0$, and pick $k$ with the property $\ell_{k+1}\leq \ell(Q)\leq \ell_k$, so that $Q$ can meet, at most, $(d+1)d^n$ cubes $Q_j^k$. Thus $\mu(Q)\leq ((d+1)d^n)^{-(k-1)}$ and we deduce
\begin{align}
	\label{eq3.3.7}
	\mu(Q)\lesssim \frac{1}{((d+1)d^n)^{(k+1)}}=\ell_{k+1}^{n+1}\leq \ell(Q)^{n+1}
\end{align}
meaning that $\mu$ presents upper $s$-parabolic $n+1$-growth. Therefore, by \cite[Chapter IV, Lemma 2.1]{Ga}, which follows from Frostman's lemma, we get $\pazocal{H}^{n+1}_{p_s}(E)\geq \pazocal{H}^{n+1}_{p_s,\infty}(E)>0$. Moreover, observe that for a fixed $0<\varepsilon\ll 1$, there is $k$ large enough so that $\text{diam}_{p_s}(Q_j^k)\leq \varepsilon$. Thus, as $E_{k,p_s}$ defines a covering of $E_{p_s}$ admissible for $\pazocal{H}^{n+1}_{p_s,\varepsilon}$, we get
\begin{align*}
	\pazocal{H}^{n+1}_{p_s,\varepsilon}(E)\leq \sum_{j=1}^{((d+1)d^n)^k}\text{diam}_{p_s}(Q^k_j)^{n+1}\simeq \ell_k^{n+1}\,((d+1)d^n)^k = 1.
\end{align*}
Since this procedure can be done for any $\varepsilon$, we also have $\pazocal{H}^{n+1}_{p_s}(E)<\infty$ and thus
\begin{equation*}
	0<\pazocal{H}_{p_s}^{n+1}(E_{p_s})<\infty.
\end{equation*}
\begin{rem}
    Let us observe that as $s\to 1/2$, the value of $d$ grows arbitrarily. That is, as $s$ approaches the value $1/2$, the Cantor set $E_{p_s}$ becomes progressively more and more dense in the unit cube. This suggests that for $s=1/2$, that corresponds to a space-time usual Lipschitz condition in $\mathbb{R}^{n+1}$, the capacity one would obtain should be comparable to the Lebesgue measure in $\mathbb{R}^{n+1}$, as in \cite{U} for analytic capacity.
\end{rem}
\begin{thm}
\label{thm3.3.5}
	Given $s\in(1/2,1]$, the Cantor set $E_{p_s}$ defined in \eqref{eq3.3.5} with the choice \eqref{eq3.3.6} is removable for Lipschitz $s$-caloric functions.
\end{thm}
\begin{proof}
	We follow an analogous proof to that of \cite[Theorem 6.3]{MPrT}. We will assume that $E_{p_s}$ is not removable and reach a contradiction. By Theorem \ref{thm3.3.1}, there is $\nu$ a distribution with $\text{supp}(\nu)\subset E_{p_s}$ satisfying $|\langle \nu,1\rangle|>0$ and
	\begin{equation*}
		\|\nabla_x P_s*\nu\|_{\infty} \leq 1,\qquad \|\partial_t^{\frac{1}{2s}} P_s*\nu\|_{*,p_s} \leq 1.
	\end{equation*}
	By \cite[Theorem 6.1]{MPrT}, that admits an almost identical proof in the $s$-parabolic context, $\nu$ admits the following representation:
	\begin{equation*}
		\nu=f\,\mu,\qquad \|f\|_{L^\infty(\mu)}\lesssim 1,
	\end{equation*}
	for some $f:E\to\mathbb{R}$ Borel function, and where $\mu:=\pazocal{H}^{n+1}_{p_s}|_{E_{p_s}}$ coincides, up to a constant (depending on $n$ and $s$), with the probability measure supported on $E_{p_s}$ that satisfies  $\mu(Q_j^k)=((d+1)d^n)^{-k}$ for all $j,k$.
	By \eqref{eq3.3.7} $\mu$ (and also $|\nu|$) has upper $s$-parabolic growth of degree $n+1$. By analogous arguments to those of Lemma \ref{lem3.3.4}, there must exist a constant $\kappa(n,s)$ so that
	\begin{equation}
		\label{eq3.3.8}
		\pazocal{P}^{s}_*\nu(\ox)\leq \kappa,\,\quad \mbox{$\ox\in\mathbb{R}^{n+1}$.}
	\end{equation}
	
	For $\ox\in E_{p_s}$, let $Q_{\ox}^k$ be the cube $Q_i^k$ containing $\ox$. Define the auxiliary operator
	\begin{equation*}
		\widetilde{\pazocal{P}}^{s}_*\nu(\ox) = \sup_{k\geq0} |\pazocal{P}^{s}_{\nu} \chi_{\mathbb{R}^{n+1}\setminus Q_{\ox}^k}. (\ox)|,
	\end{equation*}
	Since the cubes $Q_i^k$ are separated, applying the growth of $|\nu|$ as well as estimate \eqref{eq3.3.8}, 
	\begin{equation}
		\label{eq3.3.9}
		\widetilde{\pazocal{P}}^{s}_*\nu(\ox)\leq \kappa',\,\quad \mbox{$\ox\in E$,}
	\end{equation}
	for some constant $\kappa'$. We shall contradict the previous bound.
    
	To do so, pick $\ox_0\in E_{p_s}$ a Lebesgue point (with respect to $\mu$ and to $s$-parabolic cubes) of the Radon-Nikodym derivative $f=\frac{d\nu}{d\mu}$ such that $f(\ox_0)>0$. This can be done since $\nu(E)>0$.
	Given $\varepsilon>0$ small enough to be chosen below, consider a parabolic cube $Q_i^k$ containing $\ox_0$ such that
	\begin{equation*}
		\frac1{\mu(Q_i^k)} \int_{Q_i^k} |f(\bar y)- f(\ox_0)|\dd\mu(y)\leq \varepsilon.
	\end{equation*}
	Let us begin by choosing $\varepsilon$ so that $f(\ox_0)>\varepsilon$. This last condition implies that for a given $m\gg 1$ and any $k\leq h \leq k+m$, if $Q_j^h$ is contained in $Q_{\ox_0}^k$,
	\begin{align}
		\frac{\nu(Q_j^h)}{\mu(Q_j^h)} &=  \frac{1}{\mu(Q_j^h)}\int_{Q_j^h}f(\oy)\dd\mu(\oy)\nonumber\\
		&\geq f(\ox_0) - \bigg\rvert \frac{1}{\mu(Q_j^h)}\int_{Q_j^h}\big(f(\oy)-f(\ox_0)\big)\dd\mu(\oy)\bigg\rvert\nonumber \\
		&\geq f(\ox_0)-\varepsilon\frac{\mu(Q_{\ox_0}^k)}{\mu(Q_j^h)}=f(\ox_0)-((d+1)d^n)^{h-k}\varepsilon. \label{eq3.3.10}
	\end{align}
	For each $m\gg 1$, we fix the value of $\varepsilon$ (and therefore also the value of $k$) to satisfy
	\begin{equation*}
		\varepsilon<\frac{((d+1)d^n)^{-m}}{2}f(\ox_0),
	\end{equation*}
	which implies $\nu(Q_j^h)\geq f(\ox_0)\mu(Q_j^h)/2$ and hence, in particular, $\nu(Q_j^h)\geq 0$. Therefore, we also have
	\begin{align}
    \label{eq3.3.11}
		\frac{\nu(Q_j^h)}{\mu(Q_j^h)} = \bigg\rvert \frac{1}{\mu(Q_j^h)}\int_{Q_j^h}f(\oy)\dd\mu(\oy)\bigg\rvert \leq \frac{3}{2}f(\ox_0)
	\end{align}
	All in all, the previous estimates ensure that choosing $\varepsilon>0$ small enough (depending on $m$), every cube $Q_j^h$ contained in $Q_{\ox_0}^k$ with $k\leq h \leq k+m$ satisfies
	\begin{equation}
		\label{eq3.3.12}
		\frac{1}{2}f(\ox_0)\mu(Q_j^h)\leq \nu(Q_j^h) \leq \frac{3}{2}f(\ox_0)\mu(Q_j^h).
	\end{equation}
	Notice also that if we decompose $\nu$ in terms of its positive and negative variations, that is $\nu = \nu^+-\nu^-$, using $f(\ox_0)>0$ we deduce
	\begin{align}
		\nu^-(Q_{\ox_0}^k) &= \int_{Q_{\ox_0}^k}f^-(\oy)\dd\mu(\oy) = \frac{1}{2}\int_{Q_{\ox_0}^k}\big( |f(\oy)|-f(\oy) \big) \dd\mu(\oy)\nonumber\\
		&\leq \frac{1}{2}\int_{Q_{\ox_0}^k} |f(\oy)-f(\ox_0)|\dd\mu(\oy)-\frac{1}{2}\int_{Q_{\ox_0}^k} \big(f(\oy)-f(\ox_0)\big)\dd\mu(\oy)  \leq \varepsilon \mu(Q_{\ox_0}^k)
		\label{eq3.3.13}
	\end{align}
	Consider $\oz=(z_1,\ldots,z_n,u)$ one of the upper leftmost corners of $Q_{\ox_0}^k$ (that is, with $z_1$ minimal and $u$ maximal in $Q_{\ox_0}^k$). Since $\oz\in E_{p_s}$ and by definition $|\pazocal{P}^{s}_{\nu}\chi_{Q_{\oz}^k\setminus{Q_{\oz}^{k+m}}}(\oz)|=|\pazocal{P}^{s}_{\nu}\chi_{\mathbb{R}^{n+1}\setminus{Q_{\oz}^{k+m}}}(\oz)-\pazocal{P}^{s}_{\nu}\chi_{\mathbb{R}^{n+1}\setminus{Q_{\oz}^k}}(\oz)|\leq 2\widetilde{\pazocal{P}}_{\ast}^s\nu(\oz)$, we have
	\begin{equation*}
		\widetilde{\pazocal{P}}^{s}_\ast\nu(\oz) \geq \frac{1}{2}|\pazocal{P}^{s}_{\nu}\chi_{Q_{\oz}^k\setminus{Q_{\oz}^{k+m}}}(\oz)|\geq \frac{1}{2}|\pazocal{P}^{s}_{\nu^+}\chi_{Q_{\oz}^k\setminus{Q_{\oz}^{k+m}}}(\oz)|-\frac{1}{2}|\pazocal{P}^{s}_{\nu^-}\chi_{Q_{\oz}^k\setminus{Q_{\oz}^{k+m}}}(\oz)|.
	\end{equation*} 
	Observe that $\text{dist}_{p_s}(\oz,Q_{\oz}^k\setminus{Q_{\oz}^{k+m}})\gtrsim \ell(Q_{\oz}^{k+m})$ (with implicit constants depending on the non-self-intersection parameter $d=d(s)$), so we are able to estimate $|\pazocal{P}_{\nu^-}^s\chi_{Q_{\oz}^k\setminus{Q_{\oz}^{k+m}}}(\oz)|$ from above in the following way
	\begin{align*}
		|\pazocal{P}^{s}_{\nu^-}\chi_{Q_{\oz}^k\setminus{Q_{\oz}^{k+m}}}(\oz)|&\leq \int_{Q_{\oz}^k\setminus{Q_{\oz}^{k+m}}}|\nabla_xP_s(\oz-\oy)|\dd\nu^-(\oy)\leq \int_{Q_{\oz}^k\setminus{Q_{\oz}^{k+m}}}\frac{\dd\nu^-(\oy)}{|\oz-\oy|_{p_s}^{n+1}}\\
		&\lesssim \frac{\nu^-(Q_{\oz}^k)}{\ell(Q_{\oz}^{k+m})^{n+1}}\leq \varepsilon\frac{\mu(Q_{\oz}^k)}{\ell(Q_{\oz}^{k+m})^{n+1}}=\varepsilon\frac{((d+1)d^n)^{-k}}{((d+1)d^n)^{-m}\ell(Q_{\oz}^{k})^{n+1}}\\
		&=((d+1)d^n)^m\varepsilon,
	\end{align*}
	where we have used \cite[Theorem 2.2]{HeMPr} and \eqref{eq3.3.13}. 
    
	To estimate $|\pazocal{P}^{s}_{\nu^+}\chi_{Q^k_{\oz}\setminus Q^{k+m}_{\oz}}(\oz)|$ from below, we refer the reader to the proof of \cite[Theorem 2.2]{HeMPr} and \cite[Equation (2.5)]{HeMPr} to check that the first component of the kernel $\nabla_x P_s$, for $s<1$, satisfies
	\begin{equation*}
		(\nabla_xP_s)_1(\ox) \approx c_0\frac{-x_1t}{|\ox|_{p_s}^{n+2s+2}} \chi_{t>0},
	\end{equation*}
	for some constant $c_0>0$.
	Then, by the choice of $\oz$, it follows that
	\begin{equation}
		\label{eq3.3.14}
		(\nabla_xP_s)_1(\oz- \bar y) \geq 0\quad \mbox{ for all $\bar y\in Q^k_{\oz}\setminus Q^{k+m}_{\oz}$.}
	\end{equation}
	We write
	\begin{align*}
		|\pazocal{P}^{s}_{\nu^+}\chi_{Q^k_{\oz}\setminus Q^{k+m}_{\oz}}
		(\oz)| &\geq \int_{Q^k_{\oz}\setminus Q^{k+m}_{\oz}} (\nabla_xP_s)_1(\oz- \bar y)\,
		\dd\nu^+(\bar y) \\
		&= \sum_{h=k}^{k+m-1} \int_{Q^h_{\oz}\setminus Q^{h+1}_{\oz}} (\nabla_xP_s)_1(\oz- \bar y)\,
		\dd\nu^+(\bar y).
	\end{align*}
	By relation \eqref{eq3.3.14} and that for $k\leq h\leq k+m-1$, the set $Q^h_{\oz}\setminus Q^{h+1}_{\oz}$ contains a cube $Q^{h+1}_j$ such that for all $\bar y=(y_1,\ldots, y_n,\tau)$,
	\begin{equation*}
		0<y_1-z_1 \approx |\bar y-\oz|\approx\ell(Q^{h+1}_j),\qquad 0<u-\tau \approx \ell(Q^{h+1}_j)^{2s}.
	\end{equation*}
	Indeed, we might just consider the lower rightmost cube of the $h+1$ generation that is contained in $Q_{\overline{z}}^{h}$ and take advantage of the corner choice of $\overline{z}$ and the fact that $\overline{z}\notin Q_{\overline{z}}^k\setminus{Q_{\overline{z}}^{k+m}}$. Now, using also \eqref{eq3.3.12}, we deduce
	\begin{align*}
		\int_{Q_{\overline{z}}^h\setminus{Q_{\overline{z}}^{h+1}}}(\nabla_xP_s)_1(\overline{z}-\overline{y})\dd\nu^+(\overline{y}) &\geq \int_{Q_j^{h+1}}(\nabla_xP_s)_1(\overline{z}-\overline{y})\dd\nu^+(\overline{y})\gtrsim \frac{\nu^+(Q_j^{h+1})}{\ell (Q_j^{h+1})^{n+1}}\\
		&\gtrsim f(\overline{x}_0)\frac{\mu(Q_j^{h+1})}{\ell(Q_j^{h+1})^{n+1}}=f(\overline{x}_0).
	\end{align*}
	Therefore,
	\begin{equation*}
		|\pazocal{P}_{\nu^+}^s\chi_{Q^k_{\oz}\setminus Q^{k+m}_{\oz}}(\oz)| \gtrsim (m-1)\,f(\ox_0).
	\end{equation*}
	and combining this with the previous estimate obtained for $|\pazocal{P}_{\nu^-}^s\chi_{Q_{\overline{z}}^k\setminus{Q_{\overline{z}}^{k+m}}}(\overline{z})|$ we get
	\begin{equation*}
		\widetilde{\pazocal{P}}^{s}_\ast\nu(\overline{z})\gtrsim (m-1)f(\overline{x}_0)-C((d+1)d^n)^m\varepsilon
	\end{equation*}
	for some constant $C>0$. Then, choosing $m$ big enough and then $\varepsilon$ small enough, depending on $m$, the upper bound \eqref{eq3.3.9} cannot hold and we reach the desired contradiction.
\end{proof}
\section{The non-comparability of \mathinhead{\Gamma_{\Theta}}{} and \mathinhead{\gamma_{\Theta}^{1/2}}{} in the plane}
\label{sec5}
In this brief last section we would like to introduce a capacity tightly related to the capacities $\gamma_{\Theta^s,\ast}^{\sigma}$ presented in \cite[\textsection 5.3]{HeMPr} and compare it with the Lipschitz caloric capacity $\Gamma_{\Theta}$ studied in \cite{MPrT}. Recall that, given $E\subset \mathbb{R}^{n+1}$ compact set, we defined
\begin{equation*}
    \Gamma_\Theta(E):=\sup |\langle T, 1 \rangle|,
\end{equation*}
with the supremum taken among all distributions in $\mathbb{R}^{n+1}$ supported on $E$ such that
\begin{equation*}
    \|\nabla_x W\ast T\|_\infty \leq 1 \qquad \text{and} \qquad \|\partial_t^{1/2}W\ast T\|_{\ast,p_1}\leq 1.
\end{equation*}
Bearing in mind the above definition, we introduce the capacity $\gamma_{\Theta}^{1/2}$ analogously, but instead of requiring a $(1,1/2)$-Lipschitz property over the potentials, we ask for
\begin{equation*}
    \| (-\Delta)^{1/2} W\ast T\|_\infty \leq 1 \qquad \text{and} \qquad \|\partial_t^{1/2}W\ast T\|_{\ast,p_1}\leq 1.
\end{equation*}
We will prove that $\Gamma_\Theta$ and $\gamma_{\Theta}^{1/2}$ share the same critical dimension but, at least in $\mathbb{R}^2$, they do not share the same removable sets. This is remarkable since the operators $\nabla_x$ and $(-\Delta)^{1/2}$ present Fourier symbols with shared homogeneity.

Hence, let us fix $s=1$ and study the following capacity:
\begin{defn}
    Given $E\subset \mathbb{R}^{n+1}$ compact set define its \textit{$\Delta^{1/2}$-caloric capacity} as
		\begin{equation*}
			\gamma^{1/2}_{\Theta}(E):=\sup |\langle T, 1 \rangle| ,
		\end{equation*}
		where the supremum is taken among all distributions $T$ with $\text{supp}(T)\subseteq E$ and satisfying
		\begin{equation*}
			\big\|(-\Delta)^{1/2} W\ast T\big\|_{\infty} \leq 1, \hspace{0.75cm} \big\|\partial^{1/2}_t W\ast T\big\|_{\ast, p}\leq 1.
		\end{equation*}
		Such distributions will be called \textit{admissible for $\gamma^{1/2}_{\Theta}(E)$}. Notice that the operator $(-\Delta)^{1/2}$ can be represented as
\end{defn}
	\begin{equation*}
		(-\Delta)^{1/2}\varphi(x,t) \simeq \text{ p.v.}\int_{\mathbb{R}^n}\frac{\varphi(x,t)-\varphi(y,t)}{|x-y|^{n+1}} \simeq  \sum_{j=1}^n R_j\partial_j \varphi,
	\end{equation*}
	where $R_j$ are the usual $n$-dimensional Riesz transforms, with Fourier multiplier an absolute multiple of $\xi_j/|\xi|$. Observe that as a direct consequence of \cite[Theorem 5.8]{HeMPr} we obtain the following growth result for admissible distributions for $\gamma^{1/2}_{\Theta}$.
    \begin{thm}
    \label{thm4.1.1}
		Let $E\subset\mathbb{R}^{n+1}$ be compact and $T$ be an admissible distribution for $\gamma^{1/2}_{\Theta}$. Then, $T$ presents upper $1$-parabolic growth of degree $n+1$, that is,
		\begin{equation*}
			|\langle T, \varphi \rangle|\lesssim \ell(Q)^{n+1}, \qquad \text{for $Q\subset \mathbb{R}^{n+1}$ any $1$-parabolic cube and $\varphi$ admissible for $Q$}.
		\end{equation*}
	\end{thm}

    Hence, given such growth, one of the first questions that arises is if $\Gamma_{\Theta}$ and $\gamma^{1/2}_{\Theta}$ are comparable. Indeed, the following result analogous Theorem \ref{thm3.3.2} also holds in the current setting, which implies that $\Gamma_{\Theta}$ and $\gamma^{1/2}_{\Theta}$ both share critical dimension:
	\begin{thm}
		\label{thm4.1.2}
		For every compact set $E\subset \mathbb{R}^{n+1}$ the following hold:
		\begin{enumerate}
			\item[1.] $\gamma^{1/2}_\Theta(E)\leq C\,\pazocal{H}_{\infty,p_1}^{n+1}(E)$, for some dimensional constant $C>0$.
			\item[2.] If $\text{\normalfont{dim}}_{\pazocal{H}_{p_1}}(E)>n+1$, then $\gamma^{1/2}_{\Theta}(E)>0$.
		\end{enumerate}
	\end{thm}
	
	\begin{proof}
		To prove \textit{1} we proceed analogously as we have done in the proof of Theorem \ref{thm3.3.2}, using now the growth restriction given by Lemma \ref{thm4.1.1}.
        
        To prove \textit{2} we argue as in Theorem \ref{thm3.3.2}. We name $d:=\text{\normalfont{dim}}_{\pazocal{H}_{p_1}}(E)$ and assume $0<\pazocal{H}_{p_1}^d(E)<\infty$. Apply Frostman's lemma and pick a non-zero positive measure $\mu$ supported on $E$ with $\mu(B(\ox,r))\leq r^{d}$, being $B(\ox,r)$ any $1$-parabolic ball. If we prove
	\begin{equation*}
		\big\|(-\Delta)^{1/2} W\ast \mu\big\|_{\infty}\lesssim 1\hspace{0.5cm} \text{and}\hspace{0.5cm} \big\|\partial^{1/2}_t W\ast \mu\big\|_{\ast,p}\lesssim 1,
	\end{equation*}
	we will be done, because then we would have $\gamma^{1/2}_{\Theta}(E)\gtrsim \langle \mu, 1 \rangle>0$. To estimate $\partial^{1/2}_t W\ast \mu$ we apply \cite[Lemma 4.2]{HeMPr} with $\beta:=1/2$. To study the bound of $(-\Delta)^{1/2} W\ast \mu$, apply \cite[Lemma 2.2]{MPr} to deduce that for any $\ox\in\mathbb{R}^{n+1}$,
	\begin{equation*}
		|(-\Delta)^{1/2} W\ast \mu(\ox)|\lesssim \int_{E}\frac{\text{d}\mu(\oz)}{|\ox-\oy|_{p}^{n+1}}\lesssim \text{diam}_{p}(E)^{d-(n+1)}\lesssim 1,
	\end{equation*}
	and we are done.
	\end{proof}
	
	To study whether if $\Gamma_{\Theta}$ is comparable to $\gamma^{1/2}_{\Theta}$ we notice that, as a consequence of Theorem \ref{thm3.3.3}, any subset of positive $\pazocal{H}_{p_1}^{n+1}$ measure of a non-horizontal hyperplane (that is, not parallel to $\mathbb{R}^n\times\{0\}$) is not Lipschitz caloric removable. Therefore, in the planar case $(n=1)$, the vertical line segment
	\begin{equation*}
		E:=\{0\}\times [0,1], \qquad \text{that is such that $\text{dim}_{\pazocal{H}_{p_1}}(E)=n+1=2$},
	\end{equation*}
	is a first candidate to consider. To proceed, let us define a series of  operators analogous to those presented in \textsection \ref{sec3.3}. Given $\mu$, a real compactly supported Borel regular measure with upper $1$-parabolic growth of degree $n+1$, let $\pazocal{T}_\mu$ be acting on elements of $L^1_{\text{loc}}(\mu)$ as
	\begin{equation*}
		\pazocal{T}_{\mu}f(\ox):=\int_{\mathbb{R}^{n+1}}(-\Delta)^{1/2}W(\ox-\oy)f(\oy)\text{d}\mu(\oy), \hspace{0.5cm} \ox\notin \text{supp}(\mu),
	\end{equation*}
	as well as its truncated version $\pazocal{T}_{\mu,\varepsilon}f$ and maximal operator $\pazocal{T}_{\ast,\mu} f$, with analogous definitions to those found in \textsection\ref{sec3.3} for the operator $\pazocal{P}^s$.
	
	Notice that comparing \cite[Lemma 5.4]{MPrT} and \cite[Theorem 2.3]{HeMPr} with the particular choice of $s=1$ and $\beta:=1/2$, the growth-like behavior of the kernels $\nabla_xW$ and $(-\Delta)^{1/2}W$ is analogous. From this observation, the following result admits an analogous proof to that of Lemma \ref{lem3.3.4}.
	\begin{lem}
		\label{lem4.1.3}
		Assume that $\mu$ is a real Borel measure with compact support and $n+1$ upper $1$-parabolic growth with $\|\pazocal{T}\mu\|_{\infty}\leq 1$. Then, there is $\kappa>0$ absolute constant so that
	\begin{equation*}
		\pazocal{T}_{\ast}\mu (\ox)\leq \kappa, \hspace{0.5cm} \forall\, \ox\in \mathbb{R}^{n+1}.
	\end{equation*}
	\end{lem}
	
	Using the above lemma we are able to prove the following:
	
	\begin{thm}
		\label{thm2.5.5}
		The vertical segment $E:=\{0\}\times [0,1] \subset \mathbb{R}^2$ satisfies $\gamma^{1/2}_{\Theta}(E)=0$. Therefore, $\Gamma_{\Theta}$ and $\gamma^{1/2}_{\Theta}$ are not comparable in $\mathbb{R}^2$.
	\end{thm}
	\begin{proof}
		We will prove it by contradiction, i.e. by assuming $\gamma^{1/2}_{\Theta}(E)>0$. Begin by noticing that, under this last hypothesis, we would be able to find a distribution $\nu$ supported on $E$ such that
		\begin{equation*}
			\|(-\Delta)^{1/2}W\ast \nu\|_{\infty} \leq 1, \hspace{0.5cm} \|\partial_t^{1/2}W\ast \nu\|_{\ast,p} \leq 1 \hspace{0.5cm} \text{and} \hspace{0.5cm} |\langle \nu, 1 \rangle|>0.
		\end{equation*}
		By Theorem \ref{thm4.1.1} the distribution $\nu$ has upper $1$-parabolic growth of degree $n+1=2$. Therefore, since $0<\pazocal{H}_{p_1}^{2}(E)<\infty$, by \cite[Lemma 6.2]{MPrT}, we deduce that $\nu$ is a signed measure absolutely continuous with respect to $\pazocal{H}_{p_1}^{2}|_{E}$ and there exists a Borel function $f:E \to \mathbb{R}$ such that $\nu=f\, \pazocal{H}_{p_1}^{2}|_{E}$ with $\|f\|_{L^\infty(\pazocal{H}_{p_1}^{2}|_{E})}\lesssim 1$. In addition, we will also assume, without loss of generality, that $\nu(E)>0$.
        
		We shall contradict the estimate presented in Lemma \ref{lem4.1.3} by finding a point $\ox_0\in E$ such that $\pazocal{T}^\ast \nu (\ox_0)$ is arbitrarily big. To this end, firstly, we properly choose such point by a similar argument to that of the proof of \cite[Theorem 6.3]{MPrT}. Pick $\ox_0=(0,t_0)\in E$ with $0<t_0<1$ a Lebesgue point for the density $f=\dd\nu/\text{d}\pazocal{H}_{p_1}^{2}|_{E}$ satisfying $f(\ox_0)>0$, that can be done since $\nu(E)>0$. Hence, for $\varepsilon>0$ small enough, which will be fixed later on, there is an integer $k>0$ big enough so that if $B_{k,0}$ is the $1$-parabolic ball centered at $\ox_0$ with radius $r(B_{k,0})=2^{-k}$,
		\begin{equation}
			\label{eq2.5.5}
			\frac{1}{\pazocal{H}_{p_1}^{2}(B_{k,0}\cap E)}\int_{B_{k,0}\cap E}|f(\oy)-f(\ox_0)|\text{d}\pazocal{H}_{p_1}^{2}(\oy)\leq \varepsilon.
		\end{equation}
		In addition, we may assume, without loss of generality, that $k$ is big enough so that we have the inclusion $B_{k,0}\cap (\{0\}\times \mathbb{R})\subset E$. This way, the above estimate can be simply reformulated as
		\begin{equation*}
			2^{2k}\int_{t_0-2^{-2k}}^{t_0+2^{-2k}}|f(0,t)-f(0,t_0)|\dd t\leq \varepsilon.
		\end{equation*}
		In any case, let us still work with \eqref{eq2.5.5} and the assumption $B_{k,0}\cap (\{0\}\times \mathbb{R})\subset E$. Begin by fixing the value of $\varepsilon$ so that $f(\ox_0)>\varepsilon$. Proceeding as in the proof of Theorem \ref{thm3.3.5}, choosing
        \begin{equation*}
			\varepsilon<2^{-2m-1}f(\ox_0),
		\end{equation*}
        by analogous arguments to those presented in \eqref{eq3.3.10} and \eqref{eq3.3.11} (interchanging the role of $\mu$ in this case by $\pazocal{H}_{p_1}^{2}|_{B_h}$), we get that for any given $m\gg 1$, if $\varepsilon<2^{-2m-1}f(\ox_0)$,
		\begin{equation}
			\label{eq2.5.6}
			\frac{1}{2}f(\ox_0)\pazocal{H}_{p_1}^{2}(B_h\cap E)\leq \nu(B_h) \leq \frac{3}{2}f(\ox_0)\pazocal{H}_{p_1}^{2}(B_h\cap E),
		\end{equation}
        which is an analogous bound to \eqref{eq3.3.12}.		Notice also that if we decompose $\nu$ in terms of its positive and negative variations, that is $\nu = \nu^+-\nu^-$, using $f(\ox_0)>0$ we deduce as in \eqref{eq3.3.13}
		\begin{align}
			\nu^-(&B_{k,0}) \leq \varepsilon \pazocal{H}_{p_1}^{2}(B_{k,0}\cap E).
			\label{eq2.5.7}
		\end{align}
		Having fixed $\ox_0$ and the value of $\varepsilon$, we shall proceed with the proof. Observe that since $\nu$ is supported on $E$ we have
		\begin{equation*}
			|\pazocal{T}_\nu \chi_{B_{k,0}\setminus{B_{k+m,0}}}(\ox_0)|=|\pazocal{T}_\nu \chi_{E\setminus{B_{k+m,0}}}(\ox_0)-\pazocal{T}_\nu \chi_{E\setminus{B_{k,0}}}(\ox_0)|\lesssim 2\pazocal{T}_{\ast}\nu (\ox_0).
		\end{equation*}
		Therefore,
		\begin{align*}
			\pazocal{T}_\ast \nu(\ox_0) &\gtrsim \frac{1}{2}|\pazocal{T}_\nu \chi_{B_{k,0}\setminus{B_{k+m,0}}}(\ox_0)|\\
			&\geq \frac{1}{2}|\pazocal{T}_{\nu^+} \chi_{B_{k,0}\setminus{B_{k+m,0}}}(\ox_0)|-\frac{1}{2}|\pazocal{T}_{\nu^-} \chi_{B_{k,0}\setminus{B_{k+m,0}}}(\ox_0)|.
		\end{align*}
		Since $\text{dist}_p(\ox_0, B_{k,0}\setminus{B_{k+m,0}})\geq r(B_{k+m,0})$ we estimate $|\pazocal{T}_{\nu^-} \chi_{B_{k,0}\setminus{B_{k+m,0}}}(\ox_0)|$ as follows
		\begin{align*}
			|\pazocal{T}_{\nu^-} \chi_{B_{k,0}\setminus{B_{k+m,0}}}(\ox_0)|&\leq \int_{B_{k,0}\setminus{B_{k+m,0}}}|(-\Delta)^{1/2}W(\ox_0-\oy)|d\nu^-(\oy)\\
			&\lesssim \int_{B_{k,0}\setminus{B_{k+m,0}}}\frac{d\nu^-(\oy)}{|\ox_0-\oy|_p^{2}}\leq \frac{\nu^-(B_{k,0})}{r(B_{k+m,0})^{2}}\leq \varepsilon\frac{\pazocal{H}_{p_1}^{2}(B_{k,0}\cap E)}{r(B_{k+m,0})^{2}}\leq 2^{2m}\varepsilon,
		\end{align*}
		where we have used \cite[Theorem 2.3]{HeMPr} and relation \eqref{eq2.5.7}. So we are left to study the quantity $|\pazocal{T}_{\nu^+} \chi_{B_{k,0}\setminus{B_{k+m,0}}}(\ox_0)|$. Defining $f:\mathbb{R}^n\to \mathbb{R}$ as $f(z):=e^{-|z|^2/4}$, we have that
		\begin{align}
			\label{eq2.5.2}
			(-\Delta)^{1/2}f(0)\simeq 
			\text{p.v.}\int_{\mathbb{R}^n}\frac{1-e^{-|y|^2}}{|y|^{n+1}}\dd y\simeq \int_0^\infty\frac{1-e^{-r^2}}{r^{2}}\dd r=\sqrt{\pi}.
		\end{align} 
        Relation \eqref{eq2.5.2} and the fact that for each $t>0$ we have
        \begin{align*}
			(-\Delta)^{1/2}W(x,t)&\simeq t^{-n/2}(-\Delta)^{1/2}\Big[ e^{-|\cdot|^2/(4t)} \Big](x)\\ 
            &= t^{-\frac{n+1}{2}}(-\Delta)^{1/2}e^{-|x|^2/(4t)} =: t^{-\frac{n+1}{2}}(-\Delta)^{1/2}f\big( x|t|^{-1/2} \big),
		\end{align*}
        implies that, in $\mathbb{R}^{2}$,
        \begin{equation*}
            (-\Delta)^{1/2}W(0,t)\simeq t^{-1}\chi_{t>0}.
        \end{equation*}
        Hence,
		\begin{align*}
			|\pazocal{T}_{\nu^+} &\chi_{B_{k,0}\setminus{B_{k+m,0}}}(\ox_0)| = \bigg\rvert \int_{B_{k,0}\setminus{B_{k+m,0}}\cap E} (-\Delta)^{1/2}W(0,t_0-t)\dd\nu^+(0,t)\bigg\rvert\\
			&=\bigg\rvert \int_{\big[t_0-\frac{1}{2^{2k}}, \, t_0 + \frac{1}{2^{2k}}\big]\setminus{\big[t_0-\frac{1}{2^{2(k+m)}}, \, t_0 + \frac{1}{2^{2(k+m)}}\big]}} (-\Delta)^{1/2}W(0,t_0-t)\dd\nu^+(0,t)\bigg\rvert\\
			&\simeq \bigg\rvert \int_{\big[t_0-\frac{1}{2^{2k}}, \, t_0 + \frac{1}{2^{2k}}\big]\setminus{\big[t_0-\frac{1}{2^{2(k+m)}}, \, t_0 + \frac{1}{2^{2(k+m)}}\big]}} \frac{1}{|t_0-t|} \chi_{\{t_0-t>0\}} \dd\nu^+(0,t)\bigg\rvert\\
			&=\int_{\big[t_0-\frac{1}{2^{2k}}, \, t_0 - \frac{1}{2^{2(k+m)}}\big]} \frac{1}{t_0-t}\dd\nu^+(0,t)\\
			&=\sum_{h=0}^{2m+1}\int_{\big[t_0-\frac{1}{2^{2(k+m)-h-1}}, \, t_0 - \frac{1}{2^{2(k+m)-h}}\big]} \frac{1}{t_0-t}\dd\nu^+(0,t)\\
			&\geq \sum_{h=0}^{2m+1} 2^{2(k+m)-h}\cdot \nu^{+}\bigg( \bigg[ t_0-\frac{1}{2^{2(k+m)-h-1}}, \, t_0 - \frac{1}{2^{2(k+m)-h}} \bigg] \bigg)\\
			&\geq  \sum_{h=0}^{2m+1} 2^{2(k+m)-h}\cdot \frac{1}{2} f(\ox_0)\pazocal{H}_{p_1}^2\big(B_{k+m-\frac{h}{2}}\cap E\big)\simeq (2m+1)f(\ox_0),
		\end{align*}
		where for the last inequality we have used the left estimate of \eqref{eq2.5.7}. All in all, we get
		\begin{equation*}
			\pazocal{T}_\ast \nu (\ox_0) \gtrsim (2m+1)f(\ox_0)-2^{2m}\varepsilon,
		\end{equation*}
		so choosing $m$ big enough and then $\varepsilon$ small enough, we are able to reach the desired contradiction.
	\end{proof}

	\vspace{1.5cm}
	{\small
		\begin{tabular}{@{}l}
			\textsc{Joan\ Hernández,} \\ \textsc{Departament de Matem\`{a}tiques, Universitat Aut\`{o}noma de Barcelona,}\\
			\textsc{08193, Bellaterra (Barcelona), Catalonia.}\\
			{\it E-mail address}\,: \href{mailto:joan.hernandez@uab.cat}{\tt{joan.hernandez@uab.cat}}
		\end{tabular}
	}
\end{document}